\documentclass[a4paper,11pt]{article}

\usepackage[english]{babel}
\usepackage[english]{translator}
\usepackage[T1]{fontenc}
\usepackage[utf8]{inputenc}

\usepackage{amsmath}
\usepackage{amssymb}
\usepackage{amsfonts}
\usepackage{amstext}
\usepackage{amsthm}
\usepackage{bbm}

\usepackage{graphicx}
\usepackage{color}
\usepackage{subfigure}
\usepackage{float}
\usepackage{marvosym}
\usepackage{enumerate}
\usepackage{enumitem}
\usepackage{booktabs}
\usepackage{hyperref}
\usepackage[
  hmarginratio={1:1},     
  vmarginratio={1:1},     
  textwidth=460pt,        
  heightrounded,          
  bottom=3.65cm,
  top=3.65cm,
]{geometry}

\hypersetup{
  pdffitwindow=false,
  pdfhighlight=/O,
  pdfnewwindow,
  colorlinks=true,
  citecolor=red,             
  linkcolor=blue,            
  menucolor=blue,            
  urlcolor=blue,            
  pdfpagemode=UseOutlines,
  bookmarksnumbered=true,
  linktocpage,
  pdfkeywords={},
  pdfcreator={pdflatex},
  pdfproducer={LaTeX with hyperref}
}

\newcommand{\C}{\mathbb{C}}    
\newcommand{\R}{\mathbb{R}}       
\newcommand{\N}{\mathbb{N}}           

\renewcommand{\Re}{\mathrm{Re}\,}

\renewcommand{\i}{\mathrm{i}}

\newcommand{\A}{\boldsymbol{A}}
\renewcommand{\H}{\boldsymbol{H}}
 
\renewcommand{\div}{\mathrm{div} \,}		
\newcommand{\curl}{\mathrm{curl} \,}
\newcommand{\ls}{\lesssim}
\newcommand{\EGL}{E_{\mbox{\tiny GL}}} 
\newcommand{\dx}{\,\,\textnormal{d}x}

\newcommand{\argmin}[1]{\underset{#1}{\operatorname{arg}\operatorname{min}} \,}

\theoremstyle{definition}
\newtheorem{definition}{Definition}[]

\newtheorem{algorithm}[definition]{Algorithm}

\theoremstyle{plain}

\newtheorem{lemma}[definition]{Lemma}

\newtheorem{conjecture}[definition]{Conjecture}

\allowdisplaybreaks

\begin{document}
\begin{center}
{\LARGE 
On second-order optimality in the high-$\kappa$ regime of the Ginzburg–Landau model
}
\end{center}

\begin{center}
{\large Christian D\"oding\footnote[1]{Institute for Numerical Simulation, University of Bonn, D-53115 Bonn, Germany, \\ e-mail: \textcolor{blue}{doeding@ins.uni-bonn.de}.}}\\[2em]
\end{center}

\noindent
\begin{center}
\begin{minipage}{0.8\textwidth}
  {\small
    \textbf{Abstract.} We study energy minimizers of the Ginzburg–Landau (GL) free energy, a fundamental model of superconductivity. We address the high-$\kappa$ regime, the regime of a large GL parameter, in which energy minimizers exhibit vortex structures whose finite element approximations require a fine mesh resolution. This difficulty is reflected in the error analysis of discrete minimizers, which relies on a second-order optimality condition. The spectrum of the energy's second Fréchet derivative must be bounded away from zero up to symmetry. In practice, the associated spectral gap decreases rapidly with the GL parameter. This degrades the quality of the approximations because the GL parameter directly enters as an additional factor in the error estimates. Although a polynomial dependence of the spectral gap on the GL parameter has been conjectured, its precise behavior remains unclear. As a first step toward addressing this issue, we compute the spectral gap based on a finite element approximation for a range of GL parameters, providing numerical evidence for the conjectured polynomial dependence.
}
\end{minipage}
\end{center}

\vspace{12pt}

\section{Introduction}

In this work, we study the Ginzburg--Landau (GL) problem with a prescribed magnetic vector potential, which seeks a complex-valued function \(u\) that minimizes the GL energy functional
\begin{equation} \label{energy}
 E(u) = \frac{1}{2} \int_{\Omega} \Big| \frac{\i}{\kappa} \nabla u + \A u \Big|^2 + \frac{1}{2} (1 - |u|^2)^2 \, \mathrm{d}x.
\end{equation}
Here, $\Omega \subset \R^d$, $d = 2,3$, denotes a bounded domain, $\A : \Omega \to \R^d$ is a given magnetic vector potential, and $\kappa > 0$ is the GL parameter. The so-called order parameter $u$ minimizing the energy \eqref{energy} can exhibit complex pattern formation in the form of quantized vortices -- localized regions where the order parameter vanishes. The structure and density of these vortex patterns, known as Abrikosov vortex lattices \cite{Abr04}, depend sensitively on $\kappa$. In the high-$\kappa$ regime, where $\kappa$ is large, the number of vortices increases while their characteristic size decreases, creating significant challenges for accurate numerical resolution. This difficulty is closely related to the influence of $\kappa$ on the second-order optimality condition for minimizers: as $\kappa$ grows, the coercivity constant of the second derivative of the energy functional appears to decrease drastically. The exact dependence of this coercivity constant on $\kappa$ is unknown, but numerical evidence from previous works \cite{BDH25, CFH25, DDH24,DH24} suggests a polynomial decay in $\kappa$. The focus of this work is a systematic computation of the coercivity constant, leading to a numerical verification of this decay and an estimation of its rate w.r.t.~the GL parameter $\kappa$. \\

The model \eqref{energy} arises as a reduction of the GL theory of superconductivity \cite{DGP92,Landau,Tinkham}, the phenomenon of certain materials losing their 
electrical resistance at very low temperatures. The macroscopic GL theory of phase transitions describes  the state of a superconductor through a minimizer of the full GL free energy
\begin{equation} \label{full-energy}
\EGL(u,\A) = \frac12 \int_{\Omega}  
\Big| \frac{\i}{\kappa} \nabla u + \A u \Big|^2 
+ \frac12 \big(1 - |u|^2 \big)^2 + |\curl \A - \H |^2 \dx.
\end{equation}
The order parameter $u$ describing the superconducting 
properties of the material has the physical observable $|u|^2$ representing the density 
of superconducting electrons -- so-called Cooper pairs \cite{Tinkham}. In particular, $|u| = 0$ corresponds to the normal (non-superconducting) state, while $|u| = 1$ represents a perfectly superconducting state. The second variable $\A$ denotes the electromagnetic 
vector potential in the superconductor, whose curl corresponds to the magnetic field and tends to align with the externally applied magnetic field $\H$. Between the two perfect states $|u| = 0$ and $|u| = 1$, Abrikosov identified the aforementioned mixed vortex state where quantized vortices occur in the order parameter which locally destroys 
superconductivity and allows partial penetration of the external magnetic 
field, cf. \cite{Abr04}. Understanding and computing such vortex configurations is of considerable interest as it contributes to the understanding of the phenomena. As a common reduction of the model we obtain the reduced energy \eqref{energy} from \eqref{full-energy} pretending that the vector potential $\A$ is given a priori.

The numerical approximation of GL minimizers using finite element 
approximation was extensively studied in the seminal works of Du, Gunzburger, and Peterson \cite{DGP92,DuGP93}, 
where both well-posedness and optimal order approximation of \eqref{full-energy} 
in finite element spaces were established. Alternative approaches, such as 
finite volume and co-volume methods, have also been proposed \cite{Du05,DuNicolaidesWu98}. 
However, the error analyses in these works did not explicitly track the 
dependence on the material parameter $\kappa$. More recent studies \cite{BDH25,CFH25,DH24} for the reduced model \eqref{energy}, and later 
for the full model \cite{DDH24}, revealed the presence of a 
$\kappa$-dependent pollution effect. The effect imposes resolution conditions 
on the computational mesh that are required to correctly capture energy 
minimizers and, in particular, the vortex structures. These resolution 
conditions involve the aforementioned coercivity constant appearing in the 
second-order optimality conditions of the GL energy leading to an interest of its behavior in the high-$\kappa$ regime. We mention that there is also a large literature on time-dependent Ginzburg-Landau problems regarding $L^2$- or Schr\"odinger flows of the energy \eqref{energy} or \eqref{full-energy}. We exemplarily refer to \cite{CorsoKemlinMelcherStamm2025,GJX19,GaoS18,LiZ15,LiZ17} and the references therein. 

We conclude the introduction with an overview of the article. In Section~\ref{sec:GL}, we introduce the analytical framework of the GL problem and review established results on the existence and stability of GL energy minimizers. We then discuss the first- and second-order optimality conditions, introduce the coercivity constant of the second Fr\'echet derivative that is central to our study, and formulate a conjecture on its decay. Furthermore, we explain its relation to the spectral gap and to the second-smallest eigenvalue of the second derivative. In Section~\ref{sec:computation}, we present the numerical discretization and the iterative method used to compute GL minimizers, and discuss recently proven error estimates together with the pollution effect in the GL problem and its relation to the second order optimality condition. Section~\ref{sec:numerics} is devoted to the main contribution of this work. There, we present numerical experiments computing GL minimizers for several parameter settings, verify the conjectured polynomial decay of the coercivity constant in the second-order optimality condition, and provide an estimate for its rate.

\section{The Ginzburg-Landau model} \label{sec:GL}

Let $\Omega \subset \mathbb{R}^d$, $d=2,3$, be a bounded domain with Lipschitz boundary $\partial \Omega$, and denote by $L^2 = L^2(\Omega,\C)$ the space of complex-valued, square-integrable functions equipped with the real inner product $(v,w) = \Re \int_{\Omega} v \, \overline{w} \dx$ and norm $\| \cdot \|_{L^2}$. Further, denote by $H^k = H^k(\Omega, \mathbb{C})$ the standard Sobolev space of complex-valued 
functions with derivatives in $L^2$ up to order $k \in \N$ and denote by $\| \cdot \|_{H^k}$ and $|\cdot |_{H^k}$ the associated norms and semi-norms. \\
The GL problem then reads as finding a minimizer
\begin{equation} \label{minimizer}
u = \argmin{v \in H^1} E(v).
\end{equation}
Throughout the article we assume for the energy $E$ that $\kappa > 0$ and that the vector potential $\A$ is bounded, has vanishing weak divergence, and has zero-normal trace on $\Omega$, i.e., $\A \in \big( L^\infty(\Omega,\R)\big)^d$, $\div \A = 0$ in $\Omega$, and $\A \cdot \mathbf{n} = 0$ on $\partial \Omega$ where $\mathbf{n}$ is the outer normal unit vector on $\partial \Omega$. Finally, we use the notation  $a \ls b$, if there exists a constant $C > 0$ which is independent of $\kappa$ and any mesh sizes such that $a \le Cb$ holds. \\

Under these general assumptions one proves existence of a GL energy minimizer by a classical compactness argument, which also applies to the full model \eqref{full-energy} (cf.  \cite{DGP92}). Furthermore, stability bounds on $u$ can be shown that are crucial for the established analysis of the GL problem \cite{DH24,DGP92}. We note these well-known properties in the following lemma and recap the proof of existence here for the sake of completeness.

\begin{lemma}[existence and stability bounds] \label{lem:existence}
Under the general assumptions, there exists at least one minimizer $u \in H^1$ of $E$ such that \eqref{minimizer} holds and
\begin{align*}
	0 \le | u | \le 1, \quad \| \nabla u \|_{L^2} + \kappa \| u \|_{L^2} \ls \kappa.
\end{align*}
If in addition $\partial \Omega$ is of class $C^2$ or $\Omega$ is a convex polytope, then $u \in H^2$ with
\begin{align*}
	|u|_{H^2} \ls \kappa^2.
\end{align*}
\end{lemma}

\begin{proof}
Since $E(v) \ge 0$, it is bounded from below. Let $\{v_n\}_{n \in \N}$ be a minimizing sequence such that $E(v_n) \rightarrow \inf_{v \in H^1} E(v) =: E_\infty$ as $n \rightarrow \infty$. Now for any $v \in H^1$ we estimate
\begin{align*}
	\| v \|^2_{L^2} \lesssim \|1 - |v| \|^2_{L^2} + |\Omega| \lesssim \int_\Omega
	\big( 1 - |v| \big)^2 \big( 1 + |v| \big)^2 \dx + |\Omega| \lesssim  E(v) + |\Omega|
\end{align*}
and
\begin{align*}
	\| \nabla v \|^2_{L^2} \lesssim \| \nabla v +  \A  v \|^2_{L^2} + \| \A\|^2_{L^\infty} \| v \|^2_{L^2} \lesssim  E(v) + |\Omega|.
\end{align*}
Thus, $\{ v_n \}_{n \in \N}$ is bounded in $H^1$ and since $H^1$ embeds into $L^4$ compactly we can extract a subsequence, again denoted by $\{ v_n \}_{n \in \N}$, such that	$v_n \rightarrow u$ strongly in $L^4$ and $v_n \rightarrow u$ weakly in $H^1$ for some limit $u \in H^1$. Since $E$ is weakly lower semicontinuous in $H^1$ we conclude
\begin{align*}
	E(u) \le \liminf_{n \rightarrow \infty} E(v_n) = \lim_{n \rightarrow \infty} E(v_n) = E_\infty.  
\end{align*}
Thus, $u \in H^1$ is a minimizer of $E$ and \eqref{minimizer} holds. The proof of the stability bounds can be found e.g. in \cite{CFH25} and \cite{DH24}.
\end{proof}

Due to the stability bounds in Lemma \ref{lem:existence} we introduce the $\kappa$-weighted norm
\begin{align*}
	\| v \|_{H^1_\kappa}^2 := \kappa^{-2} \| \nabla v \|_{H^1_\kappa}^2 + \| v \|_{L^2}^2 
\end{align*}
which is the natural energy norm for the GL problem. Note that for the minimizer $u$ Lemma \ref{lem:existence} implies $\| u \|_{H^1_\kappa} \ls 1$.

\subsection{First- and second-order optimality condition}

It follows from classical optimization theory (e.g. \cite{Struwe,Zeidler}) that the minimizer $u$ from \eqref{minimizer} needs to be a critical point of the energy functional. In other words, let $\langle \cdot, \cdot \rangle$ denote the dual pairing on the real Hilbert space $H^1$. Then the real Fr\'echet derivative of $E$ at $u$ given by
\begin{align*}
	\langle E'(u), v \rangle = \Re \int_{\Omega} \big( \tfrac{\i}{\kappa} \nabla u + \A u \big) \overline{\big( \tfrac{\i}{\kappa} \nabla v + \A v \big)} - (1 - |u|^2) u \overline{v} \dx, \quad v \in H^1
\end{align*}
needs to vanish, i.e., $\langle E'(u) ,v \rangle = 0$ for all $v \in H^1$. This first order optimality condition is equivalent to the famous Ginzburg-Landau equation
\begin{align} \label{GLE}
	\Big(\frac{\i}{\kappa} \nabla + \A \Big)^2 u - u + |u|^2u = 0 \quad \text{in } \Omega
\end{align}
with homogeneous Neumann boundary $\nabla u \cdot \mathbf{n} = 0$ on $\partial \Omega$ which holds in strong form if $u$ is $H^2$-regular or needs to be read in weak form in general. However, this is only a necessary but no sufficient condition for the minimum. In particular, every (weak) solution to the GL equation \eqref{GLE} can be a local minimum or even a saddle-point. To ensure that the critical point is at least a local minimum, we consider the second-order optimality condition stating that the second Fr\'echet derivative $E''(u): H^1 \times H^1 \rightarrow \R$ given by
\begin{align*}
	\langle E''(u) v, w \rangle = \Re \int_{\Omega} \big( \tfrac{\i}{\kappa} \nabla v + \A v \big) \overline{\big( \tfrac{\i}{\kappa} \nabla w + \A w \big)} - (1 - |u|^2) v \overline{w} + 2 \Re(u \overline{v}) u \overline{w} \dx, \quad v,w \in H^1
\end{align*}
is positive definite. However, as we show now $E''(u)$ cannot be positive definite since there is a symmetry under complex rotations hidden in the expression of the GL energy: it holds
\begin{align*}
	E(e^{i\theta} v) = E(v) \quad \forall \theta \in [0,2\pi), \, v \in H^1. 
\end{align*}
Thus, if $u$ is a minimizer given by \eqref{minimizer} so is $e^{\i \theta} u$ for any angle $\theta \in [0, 2\pi)$. Consequently the set of minimizers forms a manifold $\mathcal{M}(u) := \{ e^{\i \theta} u : \, \theta \in [0,2\pi) \}$. Locally, this manifold admits the smooth parametrization
\begin{align*}
	\gamma: (-\varepsilon, \varepsilon) \rightarrow \mathcal{M}(u), \quad \gamma(\theta) = e^{\i \theta} u
\end{align*}
for any $0 < \varepsilon < \pi$. Its tangent space at $u$ is spanned by
\begin{align*}
	\gamma'(0) = \frac{\mathrm{d}}{\mathrm{d}\theta} e^{i\theta} u \Big|_{\theta = 0} = \i u.
\end{align*}
It follows using the first order optimality condition and $\gamma''(0) = -u \in H^1$,
\begin{align*}
	0 = \frac{\mathrm{d}^2}{\mathrm{d} \theta^2} E(\gamma(\theta)) \Big|_{\theta = 0} = \langle E''(
	\gamma(0)) \gamma'(0), \gamma'(0) \rangle + \langle E'(\gamma(0)), \gamma''(0) \rangle = \langle E''(u) \i u, \i u \rangle.
\end{align*}
Thus $\gamma'(0) = \i u \in H^1$ lies in the kernel of the second Fr\'echet derivative and hence $E''(u)$ cannot be positive definite. However, this is in general the only symmetry the energy admits. Blocking this neutral direction we can therefore expect the second order optimality condition to be satisfied on the orthogonal complement
\begin{align*}
	(\i u)^\bot := \{ v \in H^1 : \, (\i u, v) = 0 \}.
\end{align*}
Recall here that $(\cdot,\cdot)$ is the real inner product on $L^2$ so that we take the real orthogonal complement of the neutral mode $\i u$. \\

We assume that the minimizer $u$ given by \eqref{minimizer} is \textit{locally quasi-isolated}, i.e., it holds
\begin{align} \label{iso}
	\langle E''(u) v, v \rangle > 0 \quad \forall v \in (\i u)^\bot.
\end{align}
We now make the crucial observation that if the minimizer $u$ given by \eqref{minimizer} is locally-quasi isolated, then the second Fr\'echet derivative $E''(u)$ is coercive on the orthogonal complement $(\i u)^\bot$ w.r.t. the norm $\| \cdot \|_{H^1_\kappa}$.

\begin{lemma} \label{lem:coe}
Under the general assumptions let $u$ be a locally quasi-isolated minimizer given by \eqref{minimizer}. Then there exists a constant $\rho(\kappa) > 0$ such that
\begin{align*}
	\langle E''(u) v ,v \rangle \ge \rho(\kappa)^{-1} \| v \|_{H^1_\kappa}^2 \quad \text{for all } v \in (\i u)^\bot.
\end{align*}
\end{lemma}

\begin{proof}
Note that $E''(u): H^1 \rightarrow (H^1)^*$ is symmetric w.r.t.~the dual pairing $\langle \cdot, \cdot \rangle$. Denote by $0 \le \lambda_1 \le \lambda_2 \le \cdots$ the eigenvalues of $E''(u)$. We have shown above that $\lambda_1 = 0$ with eigenfunction $\i u$. Then the Courant-Fisher theorem yields
\begin{align*}
	\lambda_2 = \sup_{w \in H^1, \, \| w \| = 1} \inf_{v \in H^1, \, v \bot w} \frac{\langle E''(u) v, v \rangle}{(v,v)_{L^2}} = \inf_{v \in (\i u)^\bot} \frac{\langle E''(u) v, v \rangle}{(v,v)_{L^2}} > 0.
\end{align*}
This implies
\begin{align} \label{proof:coe1}
	\langle E''(u) v, v \rangle \ge \lambda_2 \| v \|_{L^2}^2 \quad \text{for all } v \in (\i u)^\bot.
\end{align}
Next we note that $|\tfrac{\i}{\kappa} \nabla v + \A v|^2 \ge \tfrac{1}{2 \kappa^2} |\nabla v|^2 - |\A v|^2$ and conclude the following G{\aa}rding inequality, cf. \cite[Lemma 4.1]{CFH25},
\begin{align} \label{proof:coe2}
\begin{split}
	\langle E''(u) v, v \rangle  & = \int_{\Omega} |\tfrac{\i}{\kappa} \nabla v + \A v|^2 + (|u|^2 - 1) |v|^2 + 2\Re (u \overline{v})^2 \dx \\
	& \ge \int_{\Omega} \tfrac{1}{2 \kappa^2} |\nabla v|^2 - |\A v|^2 + (|u|^2 - 1) |v|^2 \dx  \\
	& \ge \int_{\Omega} \tfrac{1}{2 \kappa^2} |\nabla v|^2 + \tfrac12 |v|^2  - \big( \|\A \|_{L^\infty}^2 + |u|^2 - \tfrac{3}{2} \big) |v|^2 \dx  \\
	& \ge \tfrac{1}{2} \| v \|_{H^1_\kappa}^2 - \big( \tfrac{3}{2} + \| \A \|_{L^\infty}^2 \big) \| v \|^2_{L^2} \quad \text{for all } v \in H^1.
\end{split}
\end{align}
Taking a convex combination of \eqref{proof:coe1} and \eqref{proof:coe2} with $\delta\in[0,1]$ gives
\begin{align*}
\langle E''(u)v,v\rangle &\ge (1-\delta)\lambda_2\|v\|_{L^2}^2 + \tfrac{\delta}{2} \|v\|_{H_\kappa^1}^2 - \delta (\tfrac{3}{2} + \|\A\|_{L^\infty}^2 ) \|v\|_{L^2}^2 \\
&= \left( \lambda_2 -\delta (\lambda_2 + \tfrac32 + \|\A\|_{L^\infty}^2 ) \right) \|v\|_{L^2}^2 + \tfrac{\delta}{2}\|v\|_{H_\kappa^1}^2 \quad \text{for all } v\in(\mathrm{i}u)^\bot.
\end{align*}
Choosing $\delta = \lambda_2 (\lambda_2 + \tfrac32 + \| \A \|_{L^\infty}^2 )^{-1}$ yields
\begin{align*}
\langle E''(u)v,v\rangle
\ge
\frac{\lambda_2}
{2\left(\lambda_2+\frac32+\|\A\|_{L^\infty}^2\right)}
\|v\|_{H_\kappa^1}^2
\quad \text{for all } v\in(\mathrm{i}u)^\bot.
\end{align*}
This proves the claim.
\end{proof}

As we will review in the next section, the coercivity constant $\rho(\kappa)^{-1}$ plays a central role in the numerical approximation of GL minimizers, due to a pollution effect that scales with $\rho(\kappa)$. Intuitively, $\rho(\kappa)$ reflects the curvature of the energy functional at the minimizers, and therefore influences not only the approximation properties of the minimizers but also typically the convergence rates of iterative solvers when computing minimizers in appropriate approximation spaces. \\
In the proof of Lemma \ref{lem:coe} we have seen that the coercivity constant $\rho(\kappa)^{-1}$ satisfies
\begin{align*}
	\rho(\kappa)^{-1} \lesssim \lambda_2(\kappa)
\end{align*}
where $\lambda_2(\kappa)$ is the second smallest eigenvalue of $E''(u)$. Since the smallest eigenvalue is $\lambda_1 = 0$, we can identify $\lambda_2(\kappa)$ with the spectral gap of $E''(u)$. Unfortunately, to the best of our knowledge, there are no analytic results about the dependence of the constant $\rho(\kappa)$ or the spectral gap $\lambda_2(\kappa)$ w.r.t. $\kappa$. However, from numerical observation it is conjectured that $\rho(\kappa) \sim \kappa^\alpha$ for some $\alpha \ge 1$. We formulate this in the following conjecture that we aim to verify numerically in Section \ref{sec:numerics}.

\begin{conjecture} \label{conjecture}
Under the general assumptions, there exists $\alpha \ge 1$ such that for all $\kappa \ge 1$
\begin{align*}
	\rho(\kappa)^{-1} \ls \kappa^\alpha
\end{align*}
where $\rho(\kappa)^{-1}$ is the coercivity constant of $E''(u)$ from Lemma \ref{lem:coe}.
\end{conjecture}

We note that the spectral properties of the linear part of $E''(u)$ have been studied extensively e.g. in \cite{ForH10}. The linear part is given by the so-called magnetic Neumann Laplacian $(\tfrac{\i}{\kappa} \nabla + \A)^2$. For bounded smooth domains it is shown in \cite[Theorem 8.1.1 \& 9.1.1]{ForH10} that its smallest eigenvalue $\lambda(\kappa)$ satisfies $\lambda(\kappa) \sim \kappa^{-1}$ if $\curl \A$ does not vanish in $\Omega$. This suggests a decay rate $\alpha \ge 1$ in Conjecture \ref{conjecture}. However, since this concerns only the linear part of $E''(u)$, it does not imply a corresponding bound for the full operator as the nonlinear effects might dominate the spectrum of $E''(u)$.

\section{Computation of GL minimizers} \label{sec:computation}

In this section, we introduce a discretization and an iterative numerical method for approximating minimizers of the GL energy. For the spatial discretization, we employ classical $\mathcal{P}_2$ Lagrange finite elements. Let $\mathcal{T}_h$ be a quasi-uniform, shape-regular triangulation of $\Omega$ with mesh size $h > 0$. We define the space of quadratic, $H^1$-conforming Lagrange finite elements by
\[
V_h := \{ v_h \in C(\overline{\Omega}) : v_h|_{T} \in \mathcal{P}_2(T) \ \text{for all } T \in \mathcal{T}_h \},
\]
where $\mathcal{P}_2(T)$ denotes the space of quadratic polynomials on $T \in \mathcal{T}_h$. A discrete minimizer $u_h \in V_h$ of the GL energy $E$ is then characterized by
\begin{equation} \label{discrete_minimizer}
u_h = \argmin{v_h \in V_h} E(v_h).
\end{equation}
To achieve optimal approximation order of the minimizer $u$ in the $\mathcal{P}_2$ finite element space we need to ensure $H^3$ regularity of $u$. This can be guaranteed if we assume, in addition, that $\mathbf{A} \in (W^{1,\infty}(\Omega))^d$ and that $\partial\Omega$ is of class $C^3$ or that $\Omega$ is a cuboid. Then it is shown in \cite{CFH25} for the smooth boundary and in \cite{DDH24} for the cuboid that the exact minimizer satisfies $u \in H^3(\Omega)$ and $|u|_{H^3} \lesssim \kappa^3$. Standard interpolation estimates (cf.\ \cite{BrennerScott}) then imply for the best-approximation error
\begin{align*}
	\inf_{v_h \in V_h} \| u - v_h \|_{H^1_\kappa}
\lesssim 
\kappa^{-1} h^2 |u|_{H^3} + h^2 |u|_{H^2}
\lesssim 
(\kappa h)^2.
\end{align*}
Consequently, obtaining meaningful approximations in $V_h$ requires at least the resolution condition $h \lesssim \kappa^{-1}$. However, previous works reveal that the situation is more restrictive when considering the error of the discrete minimizer from \eqref{discrete_minimizer}. In \cite[Theorem 3.2]{CFH25}, it is shown that if
\begin{equation} \label{res_cond}
h \lesssim \rho(\kappa)^{-1/2}\kappa^{-1-d/2},
\end{equation}
then the optimal convergence rate
\begin{align*}
\| u - u_h \|_{H^1_\kappa}
\lesssim 
(\kappa h)^2
\end{align*}
is achieved. In particular, there is a pollution effect present that needs to be resolved by the constraint \eqref{res_cond} to achieve optimal order of convergence. Similar results involving resolution conditions that quantify the pollution effect have been derived for general Lagrange finite elements of order $p$ (see \cite{DH24,CFH25}) as well as for multiscale approximation spaces in the spirit of the Localized Orthogonal Decomposition, cf. \cite{BDH25,MP14}. As shown in \cite{DDH24}, the pollution effect is also observed in the order parameter for minimizers of the full GL energy \eqref{full-energy}. \\
Assuming that Conjecture~\ref{conjecture} holds, the GL parameter $\kappa$ has a decisive influence on the mesh size required for accurate approximations. For the quadratic finite element approximation considered here, the resolution condition can be given explicitly in terms of $\kappa$ by
\begin{align*}
	h \lesssim \kappa^{-1-(\alpha + d)/2},
\end{align*}
where $\alpha$ denotes the rate appearing in Conjecture~\ref{conjecture}. Finally, we note that the pollution effect is indeed observed numerically through a pre-asymptotic regime for coarse mesh sizes where no error decay is observed. This gives rise to the high interest of the magnitude of $\alpha$ in order to further quantifying the pollution effect.

\subsection{A gradient descent method}

To verify Conjecture \ref{conjecture} and compute the magnitude of $\alpha$ via the spectrum of the second derivative $E''(u)$, we first need to compute the discrete minimizer satisfying \eqref{discrete_minimizer}. This is typically achieved using iterative descent methods of the form
\begin{equation}
	u^{n+1} = u^n + \tau_n d^n,
\end{equation}
where $u^n$ is the current iterate, $\tau_n > 0$ is a step size, and $d^n$ is the descent direction. In this work, we employ a nonlinear conjugate Sobolev gradient method, initially introduced for the GL problem in \cite{CFH25} and further used in \cite{DoHe25}.  
In this method, the descent direction $d^n$ incorporates the Sobolev gradient $\nabla_{X^n} E(u^n)$, defined via
\begin{equation*}
	(\nabla_{X^n} E(u^n), v)_{X^n} = \langle E'(u^n), v \rangle \quad \text{for all } v \in V_h,
\end{equation*}
for a suitably chosen metric $(\cdot,\cdot)_{X^n}$ on $H^1$. The metric is adaptively updated with respect to the current iterate $u^n$ and is given by
\begin{equation*}
	(v,w)_{X^n} := \Re \int_{\Omega} \Big( \frac{\i}{\kappa} \nabla v + \A v \Big) \overline{\Big( \frac{\i}{\kappa} \nabla w + \A w \Big)} + (|u^n|^2 + |\A|^2) v \overline{w} \, \mathrm{d}x, \quad v,w \in V_h.
\end{equation*}
The Sobolev gradient can then be computed as
\begin{equation*}
	\nabla_{X^n} E(u^n) = u^n - \mathcal{A}_n^{-1} \big( (1 + |\A|^2) u^n \big),
\end{equation*}
where $\mathcal{A}_n^{-1} v \in V_h$ is the solution to
\begin{equation*}
	(\mathcal{A}_n^{-1} v, w)_{X^n} = (v, w)_X \quad \forall w \in V_h.
\end{equation*}
Then, the descent direction for the nonlinear conjugate Sobolev gradient method is defined by
\begin{equation*}
	d^n = - \nabla_{X^n} E(u^n) + \beta_n d^{n-1},
\end{equation*}
where $\beta_n$ is the Polak--Ribi\`ere parameter \cite{PR69},
\begin{equation} \label{parameter}
	\beta_n := \max \Big\{0, 
    \frac{(\nabla_{X^n} E(u^n), \nabla_{X^n} E(u^n) -\nabla_{X^{n-1}} E(u^{n-1}))_{X^n}}
         {(\nabla_{X^{n-1}} E(u^{n-1}), \nabla_{X^{n-1}} E(u^{n-1}))_{X^{n-1}}} 
    \Big\}
\end{equation} 
and where we formally set $d^0 = 0$. The step size $\tau_n > 0$ is obtained through minimizing the energy along the descent direction, i.e.,
\begin{align} \label{stepmin}
	\tau_n = \argmin{\tau > 0} E(u^n + \tau d^n).
\end{align}
The iteration is continued until a given tolerance $\mathrm{tol} > 0$ is achieved with respect to the energy difference between successive iterates. This procedure leads to the following algorithm.  

\begin{algorithm} \label{algo}
Given an initial value $u^0 \in V_h$, a tolerance $\mathrm{tol} > 0$, and set $d^0 = 0$. Then perform the following steps for $n \in \mathbb{N}$ until termination:

\begin{enumerate}
\item \textbf{Solve} for $\delta^n \in V_h$ such that $
	(\delta^n, w_h)_{X^n} = ((1 + |\A|^2) u^n, w_h)$ for all $w_h \in V_h$.
\item \textbf{Set} the descent direction	$d^n = \delta^n - u^n + \beta_n d^{n-1}$, where $\beta_n$ is defined in \eqref{parameter}.
\item \textbf{Compute} the step size $\tau_n$ using a line search according to \eqref{stepmin}.
\item \textbf{Update} the iterate $u^{n+1} = u^n + \tau_n d^n$.
\item \textbf{Terminate} if $E(u^{n}) - E(u^{n+1}) < \mathrm{tol}$.
\end{enumerate}
\end{algorithm}

Due to the construction of the algorithm as a descent method, the iterates are expected to be energy diminishing. Moreover, the algorithm terminates when the classical $L^2$-gradient $E'(u^n)$ at the current iterate becomes sufficiently small. Consequently, the iterate $u^n$ can be regarded as a suitable approximation of a critical point of the energy functional $E$ in $V_h$. In order to verify that this critical point corresponds to at least a local minimum of the energy functional, we compute the spectrum of the second derivative. If the second smallest eigenvalue is smaller in magnitude than a prescribed tolerance $\mathrm{tol}_{\mathrm{EV}}$, we assume that the critical point corresponds to a saddle point. In this case, we perform an additional descent step in the direction of the eigenfunction associated with the second smallest eigenvalue. This perturbs the system and the algorithm is restarted from the perturbed iterate. The procedure is repeated until the obtained critical point satisfies our criterion for a local minimum, namely that the spectral gap exceeds the tolerance $\mathrm{tol}_{\mathrm{EV}}$. \\
We emphasize that a rigorous convergence analysis of the nonlinear conjugate Sobolev gradient descent method is an open problem and remains an interesting topic for future research. In particular, it is of interest to understand how the parameter $\kappa$ influences the convergence rate toward a critical point. As typical for such iterative methods, it is expected that the coercivity constant $\rho(\kappa)^{-1}$ plays a significant role, as it provides, loosely speaking, a measure of the curvature of the energy at the critical point. Numerical experiments indicate that, with increasing $\kappa$, the number of iterations required to reach the prescribed tolerance grows significantly. This provides an additional motivation to study the dependence of $\rho(\kappa)^{-1}$ on $\kappa$.

\section{Numerical results} \label{sec:numerics}

In this section we compute GL minimizers numerically in several model settings and estimate the coercivity constant $\rho(\kappa)^{-1}$ via the second smallest eigenvalue $\lambda_2(\kappa)$ of the second Fr\'echet derivative $E''(u)$. This allows us to verify Conjecture~\ref{conjecture} and to estimate the resulting decay rate $\alpha$.

For the model problems we fix the domain $\Omega = (0,1)^2$ and consider two different vector potentials $\A$: The first vector potential $\A_1$ is defined by
\begin{align*}
\A_1(x) =
\sqrt{2}
\begin{pmatrix}
\sin(\pi x_1)\cos(\pi x_2) \\
-\cos(\pi x_1)\sin(\pi x_2)
\end{pmatrix}.
\end{align*}
This vector potential belongs to the model case of an external magnetic field given by 
\begin{align*}
	\H_1(x) = 2\sqrt{2}\pi \sin(2\pi x_1) \sin(2\pi x_2)
\end{align*}
as it solves $\curl \A_1 = \H_1$, $\div \A_1 = 0$ in $\Omega$ in alignment with reduction procedure of the full GL energy \eqref{full-energy} to the reduced energy \eqref{energy}. In a similar fashion, we choose a constant homogeneous magnetic field $\H_2(x) = 2\sqrt{2}\pi $ for the second vector potential $\A=\A_2$ and define it as the solution of
\begin{align*}
\curl \A_2 = \H_2,
\qquad
\div \A_2 = 0 \qquad \text{in } \Omega.
\end{align*}
Note that, in practice, $\A_2$ is computed with a $\mathcal{P}_2$ finite element approximation on $\mathcal{T}_h$. For each model problem defined by one of the vector potentials we vary the GL parameter according to $\kappa = 5j$, $j=2,\dots,20$, so that $\kappa$ ranges from $10$ to $100$ in increments of size $5$. For each parameter configuration we employ $\mathcal{P}_2$ finite elements on a uniform triangular mesh with mesh size $h=2^{-9}$ leading to $1050625$ degrees of freedom (over $\C$). On $V_h$ we compute the discrete GL energy minimizers using the nonlinear conjugate Sobolev gradient descent method from Algorithm~\ref{algo} up to a tolerance of $\mathrm{tol}=10^{-14}$. To ensure that the obtained solutions correspond to local minimizers, we apply the test for a local minimum as described after Algorithm~\ref{algo}, using the tolerance $\mathrm{tol}_{\mathrm{EV}}=10^{-8}$.

Since this procedure only guarantees the computation of local minimizers and the nonlinear conjugate Sobolev gradient method is typically sensitive to the choice of initial values, we run the algorithm for each parameter configuration with five different initial states. These are generated from empirically chosen functions (cf.~\cite{DoHe25}),
\begin{alignat*}{2}
	\psi_1(x,y) &= \tfrac{1}{\sqrt{2}}(1+\mathrm{i}) e^{-(x^2+y^2)}, \quad & \psi_2(x,y) = \tfrac{1}{\sqrt{\pi}}\!\left(\tfrac{2}{3}(x+\mathrm{i}y)+\tfrac{1}{2}\right) e^{-(x^2+y^2)}, \\
	\psi_3(x,y) &= \tfrac{1}{\sqrt{2}}(1+\mathrm{i}) e^{10\mathrm{i}(x^2+y^2)}, \quad & \psi_4(x,y) = \tfrac{1}{\sqrt{2}}(1+\mathrm{i})(x+\mathrm{i}y) e^{10\mathrm{i}(x^2+y^2)}, \\
	\psi_5(x,y) &= \tfrac{1}{\sqrt{2}}(1+\mathrm{i}).
\end{alignat*}
The initial values are prescribed by $\varphi_j = \psi_j \circ \chi$, where the transformation $\chi : (0,1)^2 \to (-1,1)^2$ is defined by $\chi(x,y) = (2x-1,\, 2y-1)$. The energy levels of the obtained local minimizers are then compared, and the local minimizer with the lowest energy is selected. This minimizer serves as a reference global minimizer, and we compute the spectral gap $\lambda_2(\kappa)$ to validate Conjecture~\ref{conjecture}. We note that the present numerical study is restricted to variations of the GL parameter $\kappa$ and the vector potential $\A$ within the considered model settings. Whether the observed decay behavior of the spectral gap persists for more general domain geometries or other model settings is beyond the scope of the present work and remains an open question for future numerical and analytical investigations. \\

\begin{figure}[h!]
\centering
\begin{minipage}{0.19\textwidth}
\includegraphics[scale=0.17]{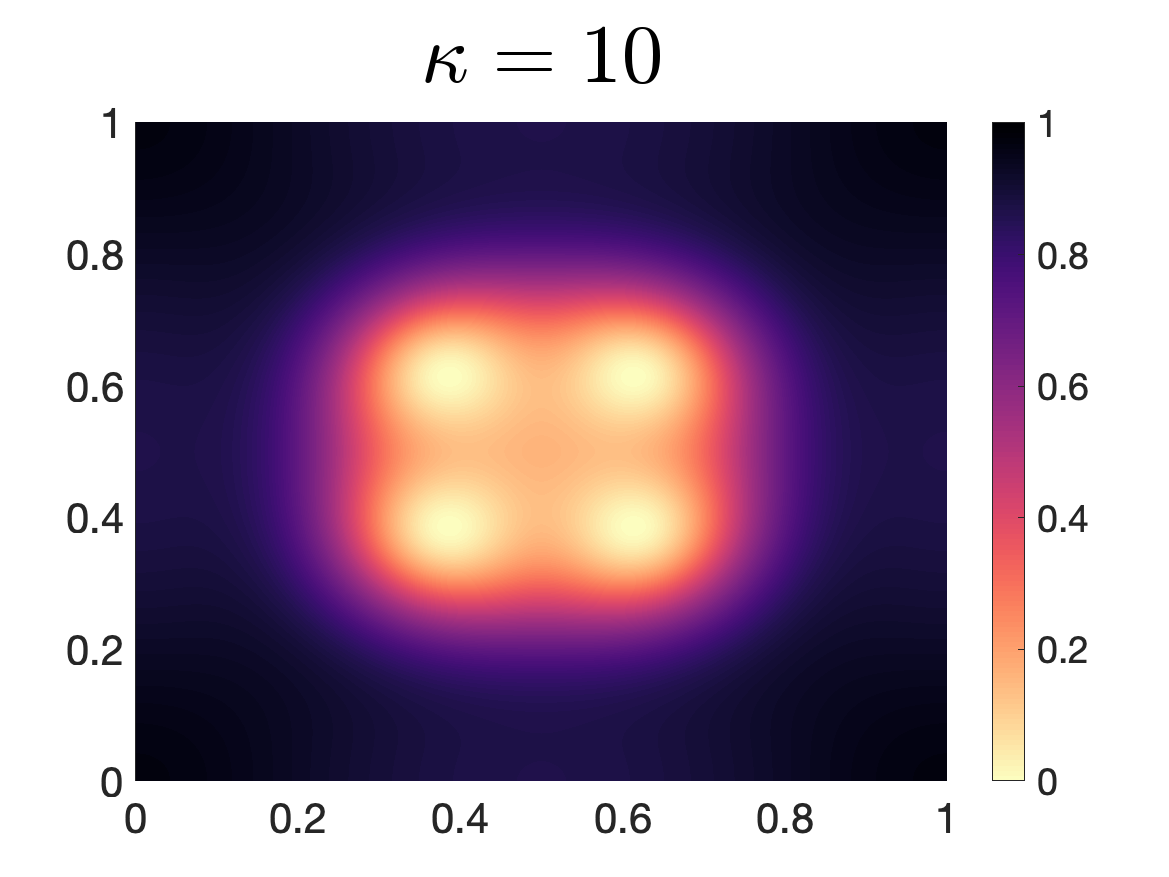}
\end{minipage}
\begin{minipage}{0.19\textwidth}
\includegraphics[scale=0.17]{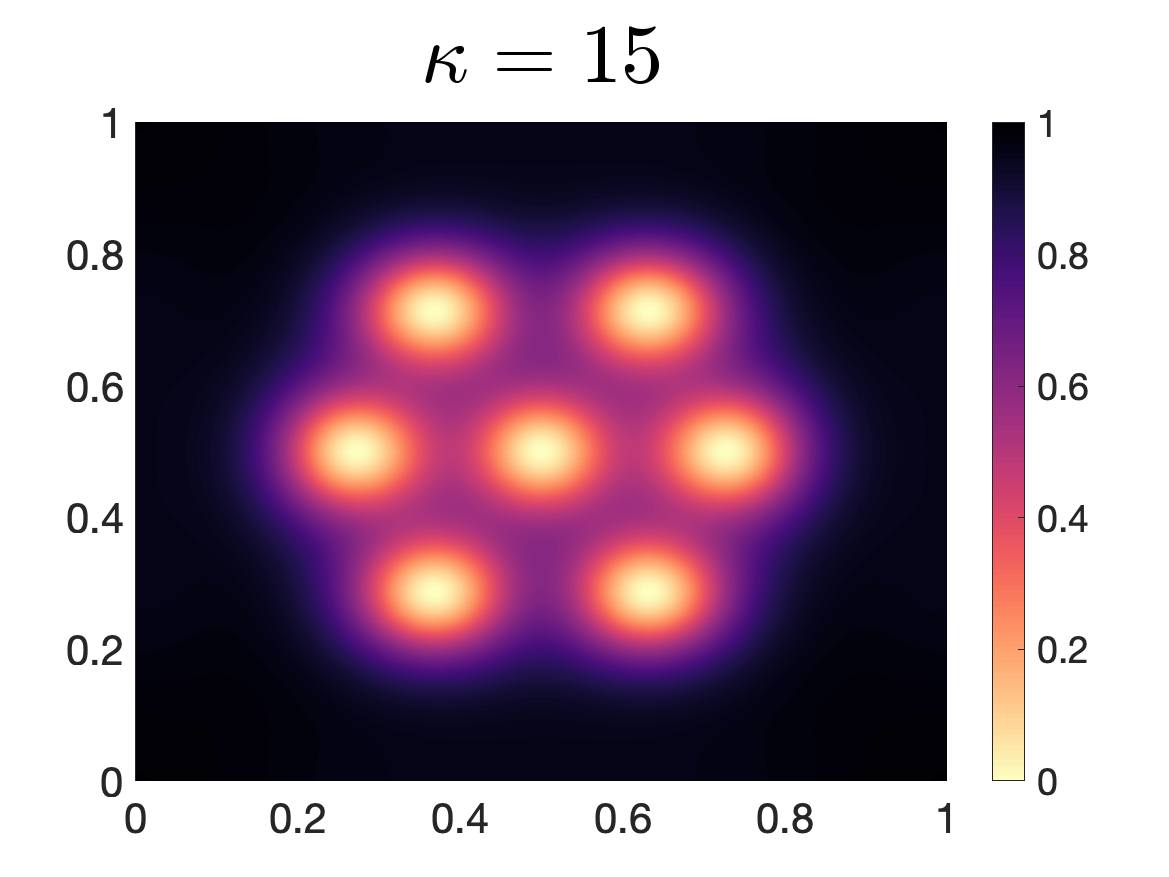}
\end{minipage}
\begin{minipage}{0.19\textwidth}
\includegraphics[scale=0.17]{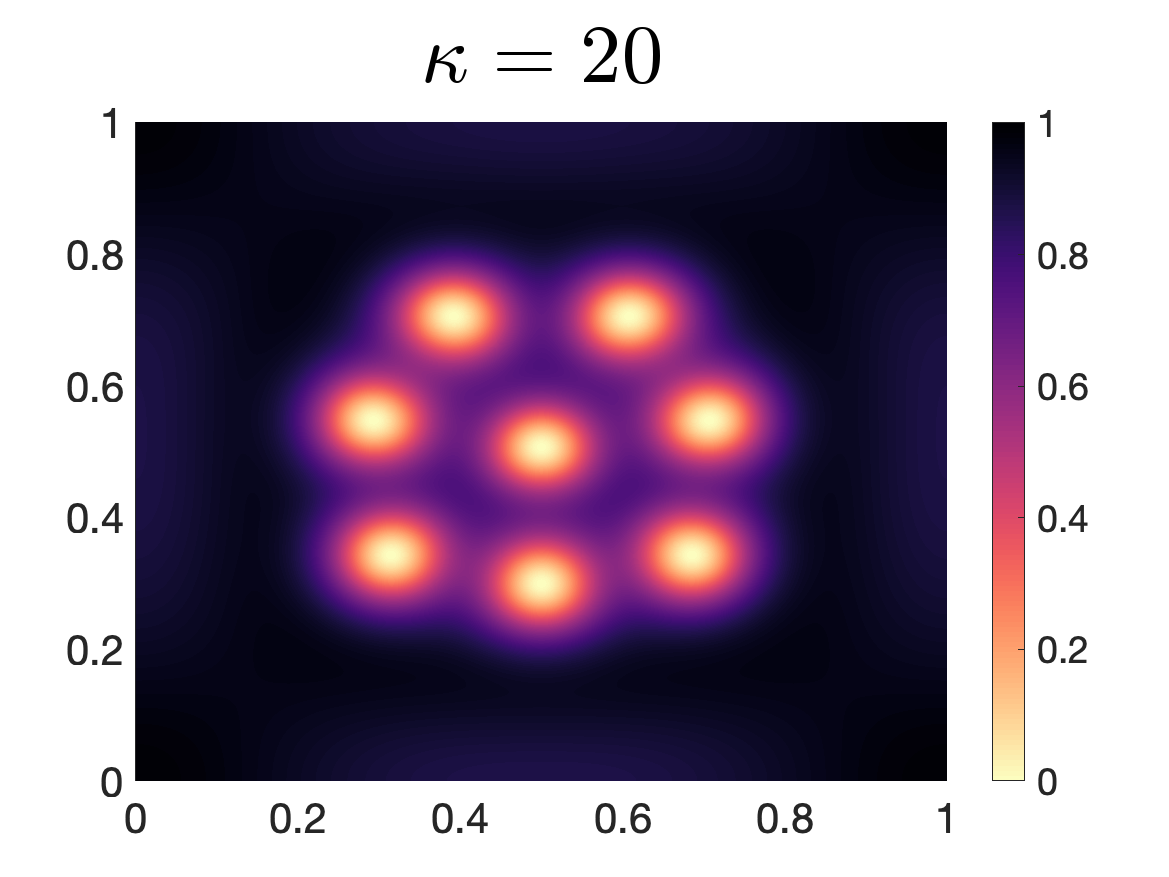}
\end{minipage}
\begin{minipage}{0.19\textwidth}
\includegraphics[scale=0.17]{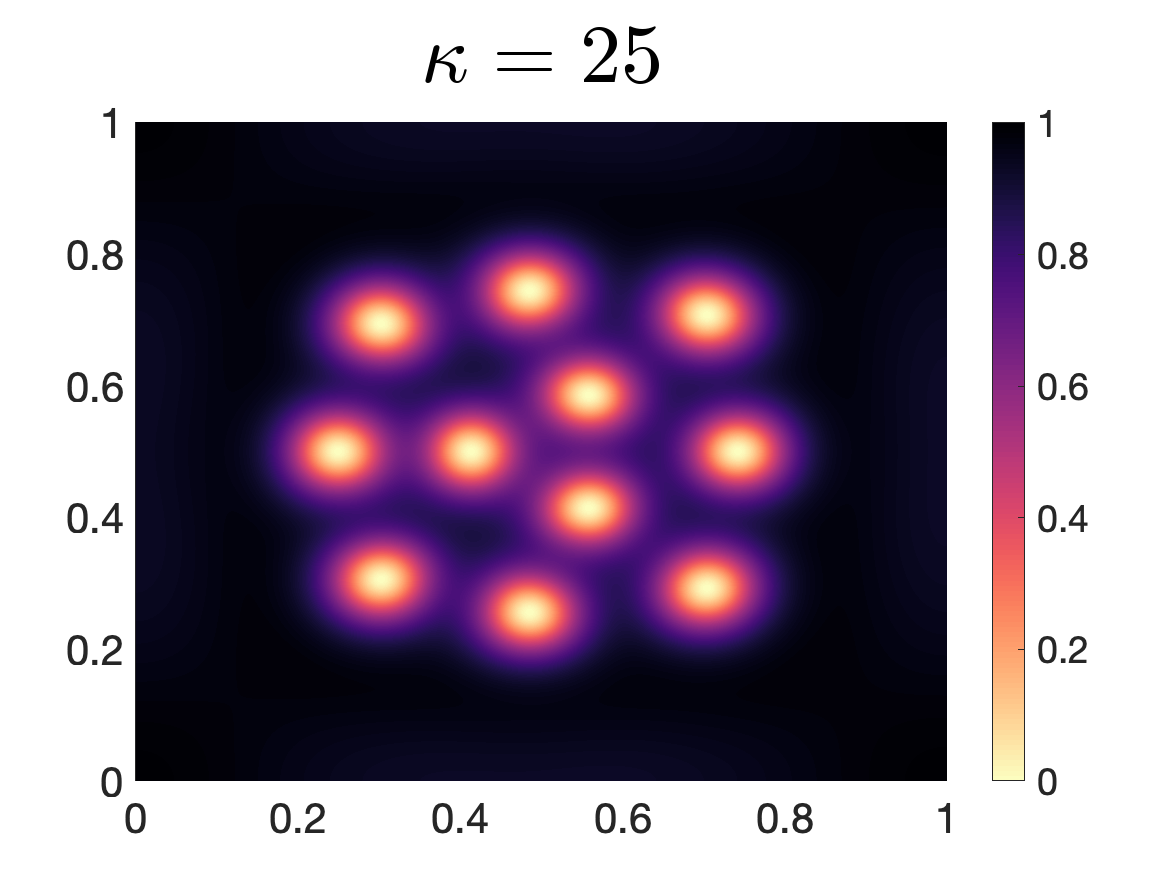}
\end{minipage}
\begin{minipage}{0.19\textwidth}
\includegraphics[scale=0.17]{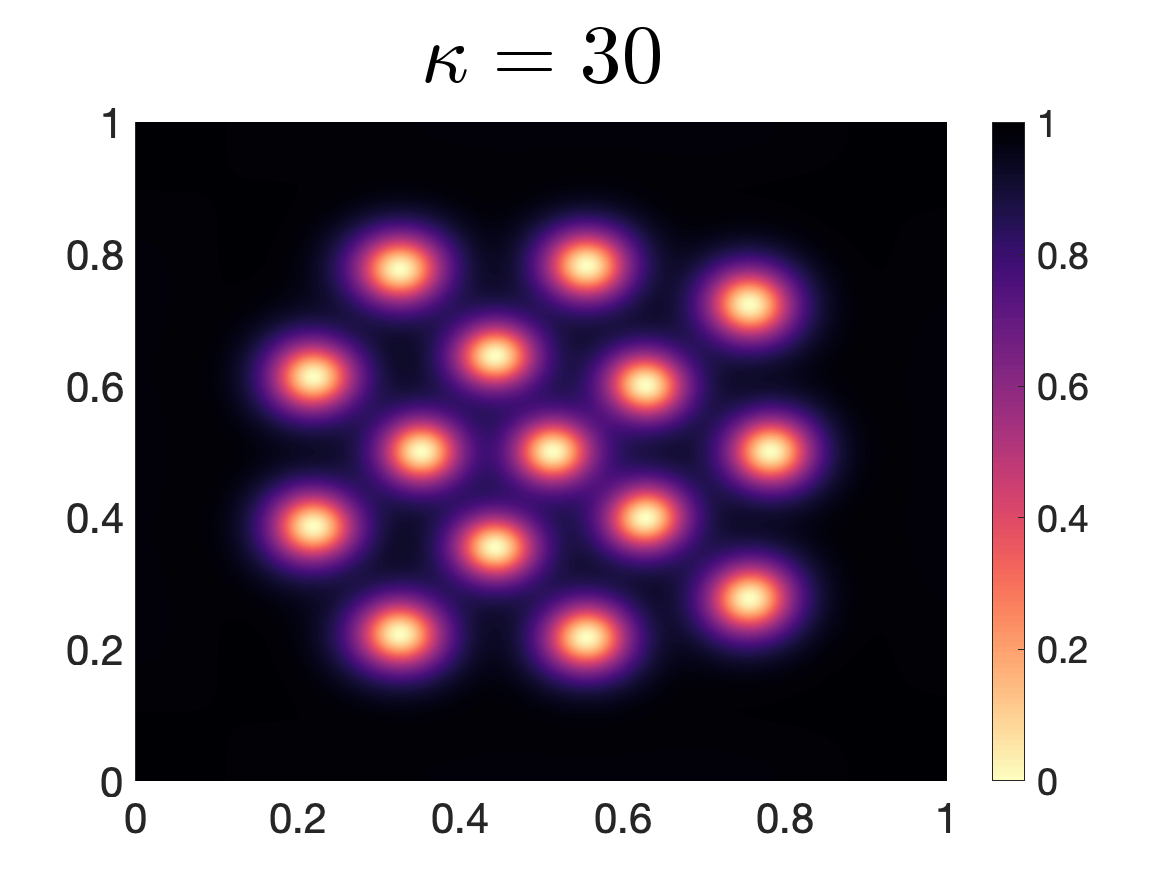}
\end{minipage} \\
\begin{minipage}{0.19\textwidth}
\includegraphics[scale=0.17]{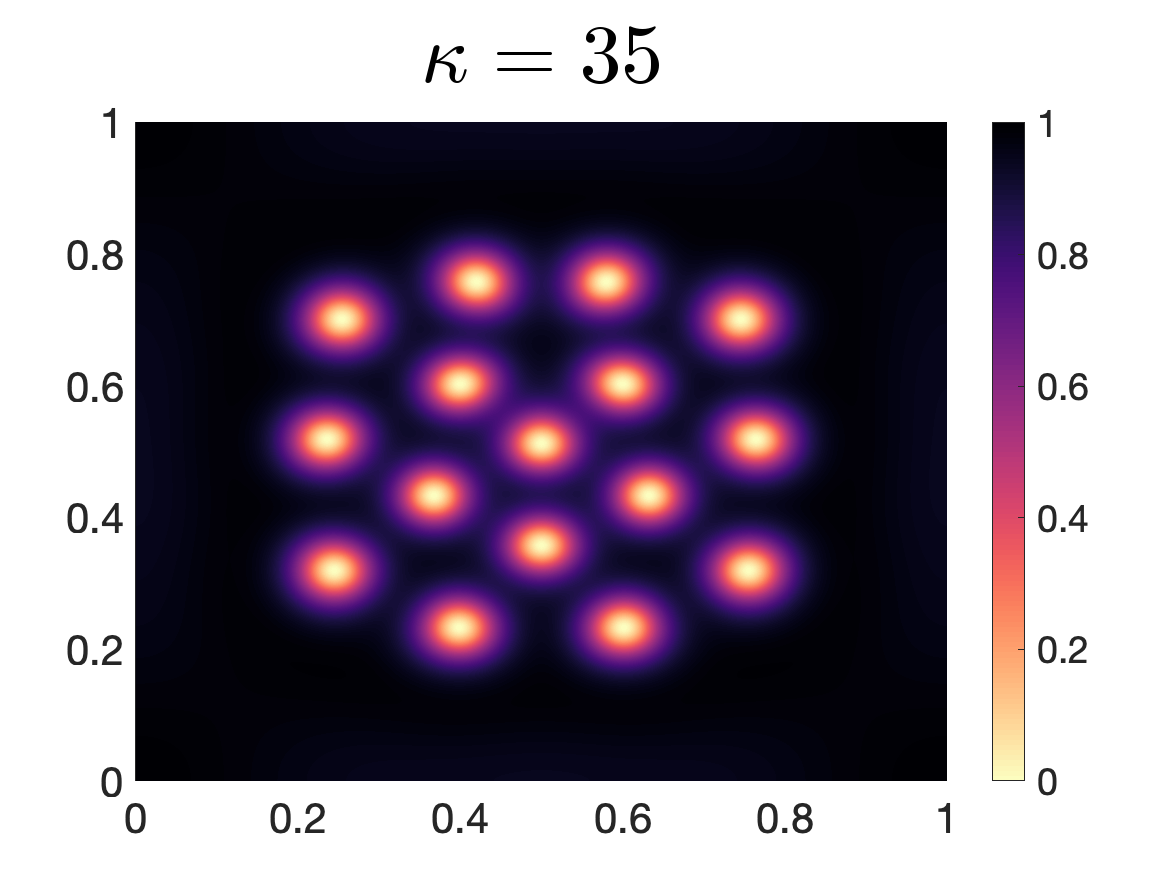}
\end{minipage}
\begin{minipage}{0.19\textwidth}
\includegraphics[scale=0.17]{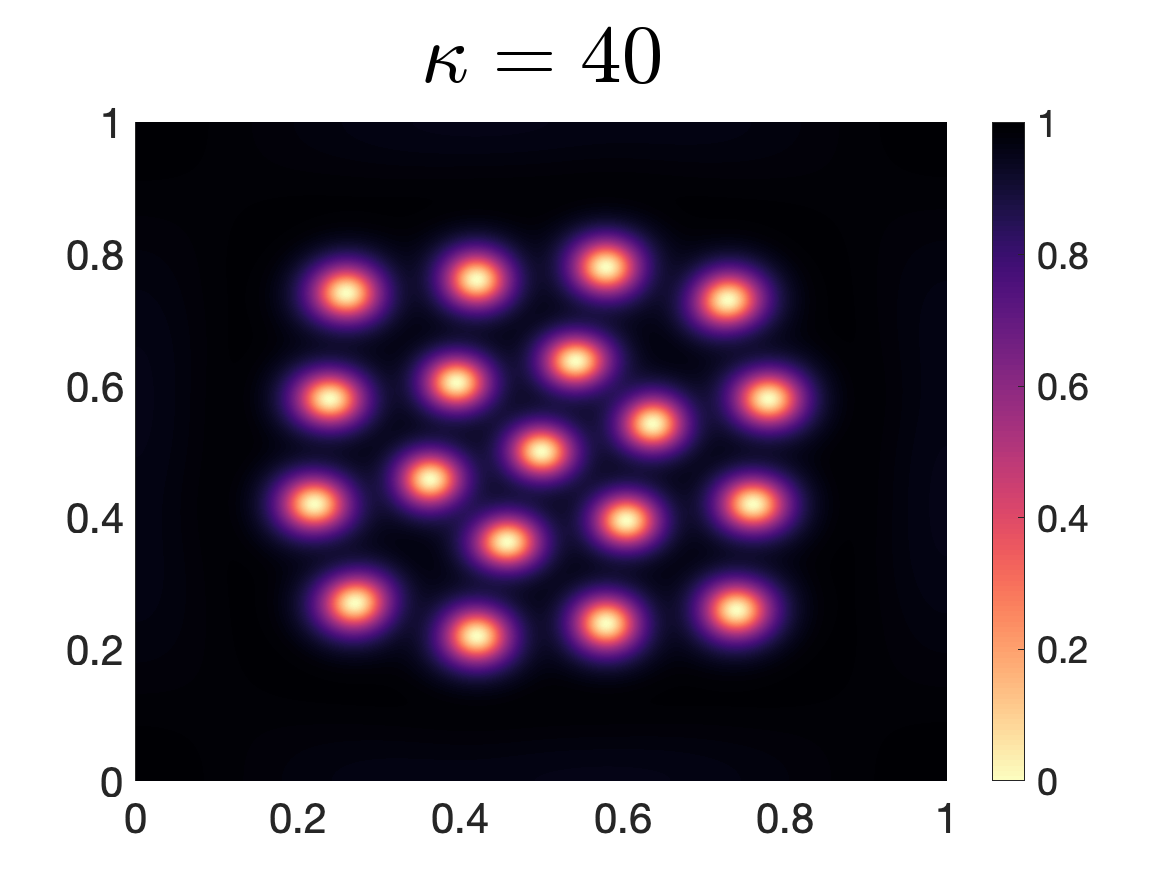}
\end{minipage}
\begin{minipage}{0.19\textwidth}
\includegraphics[scale=0.17]{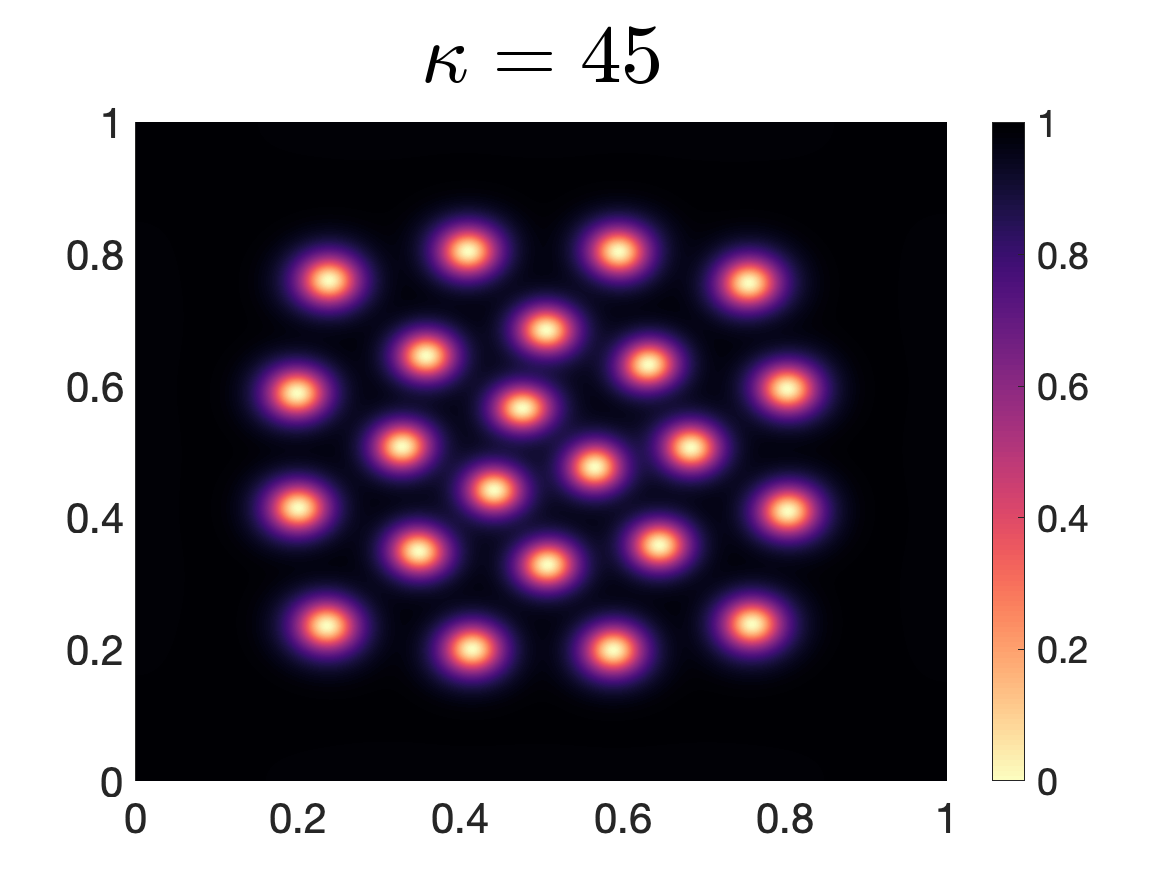}
\end{minipage}
\begin{minipage}{0.19\textwidth}
\includegraphics[scale=0.17]{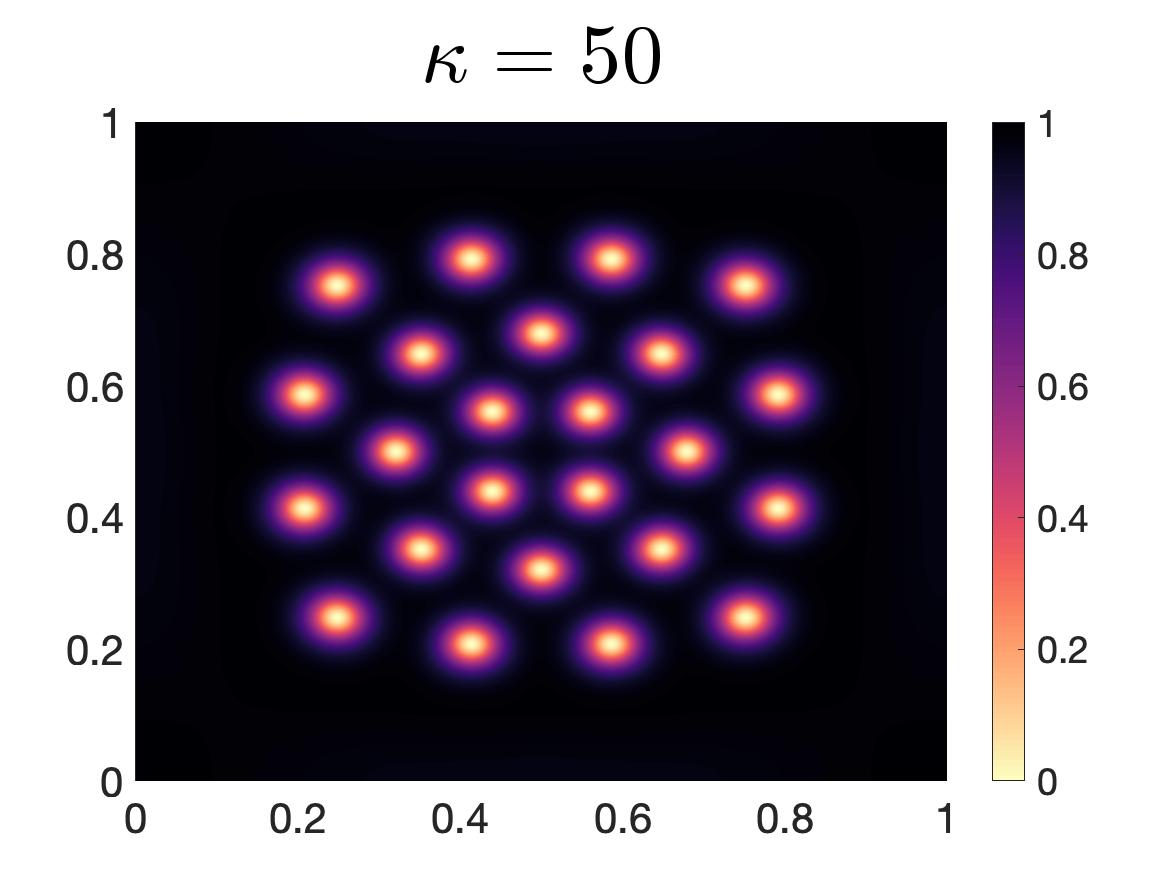}
\end{minipage}
\begin{minipage}{0.19\textwidth}
\includegraphics[scale=0.17]{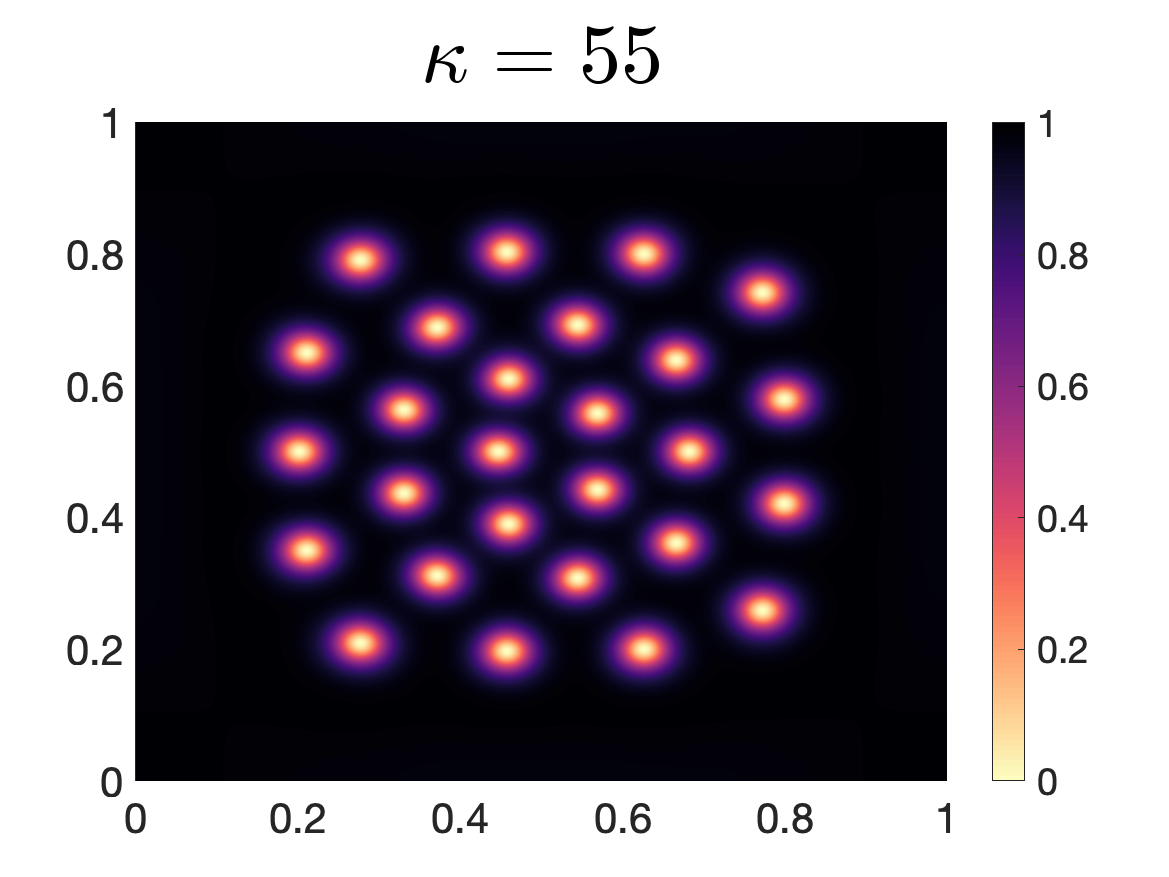}
\end{minipage} \\
\begin{minipage}{0.19\textwidth}
\includegraphics[scale=0.17]{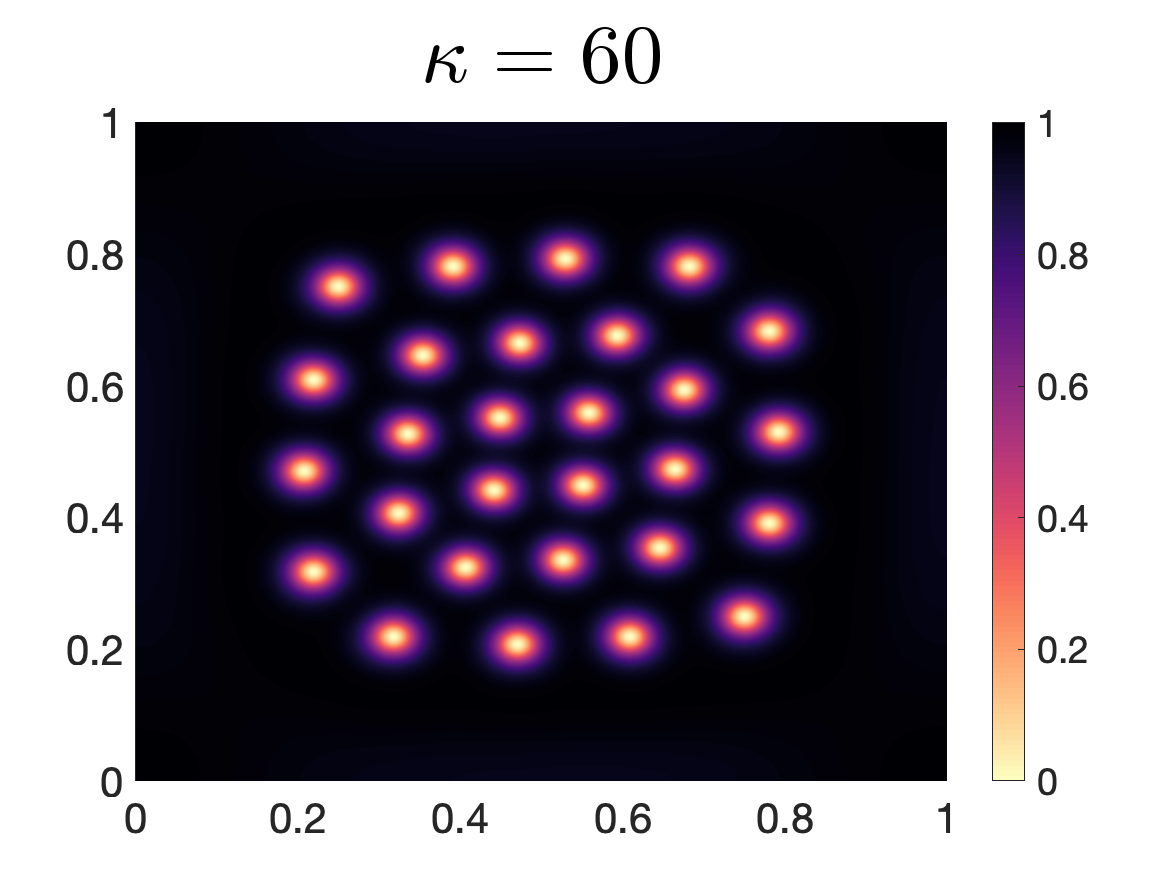}
\end{minipage}
\begin{minipage}{0.19\textwidth}
\includegraphics[scale=0.17]{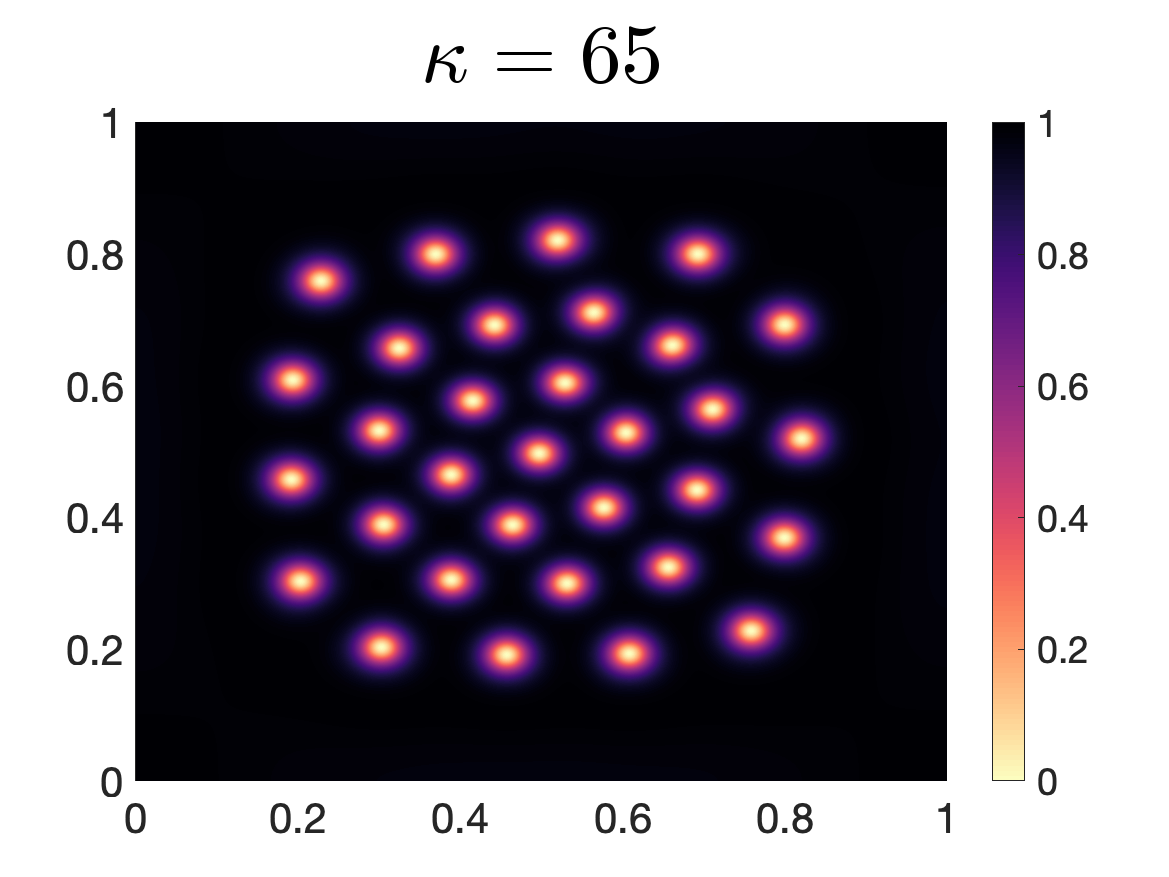}
\end{minipage}
\begin{minipage}{0.19\textwidth}
\includegraphics[scale=0.17]{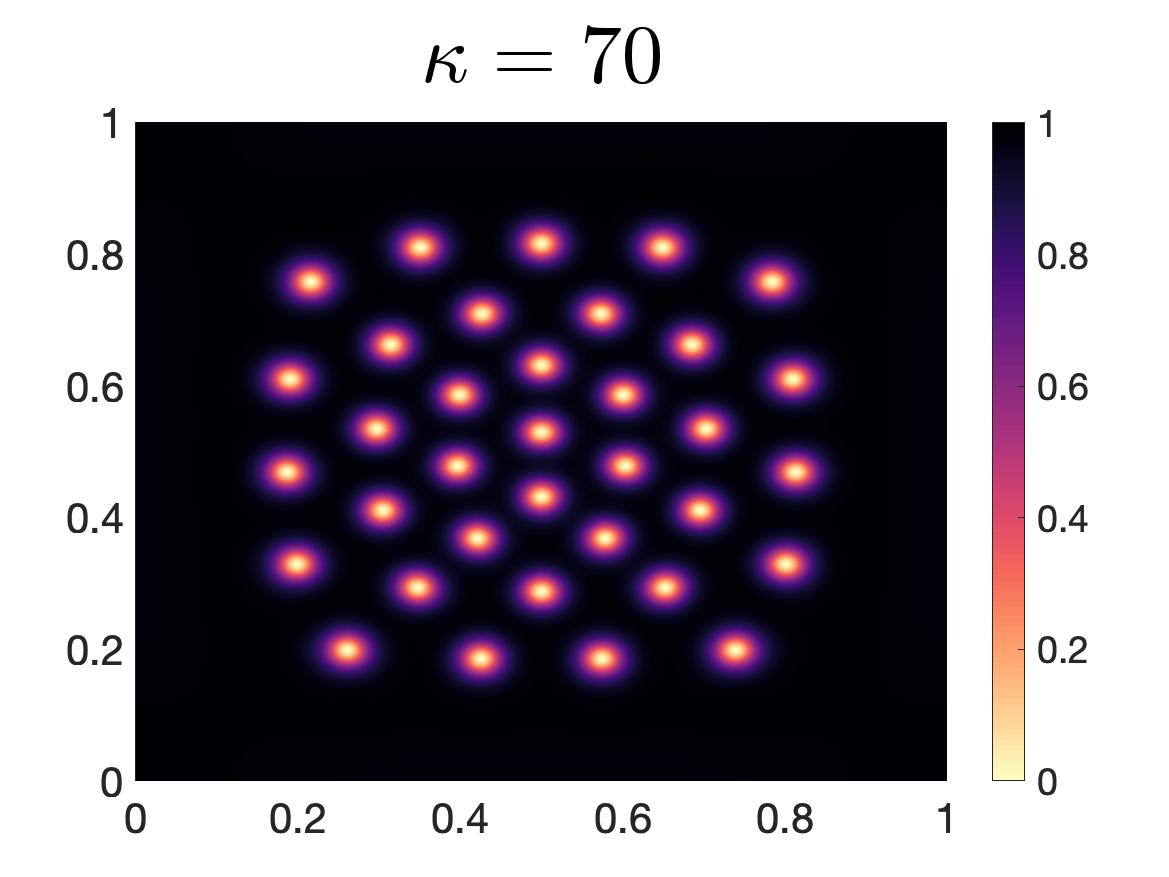}
\end{minipage}
\begin{minipage}{0.19\textwidth}
\includegraphics[scale=0.17]{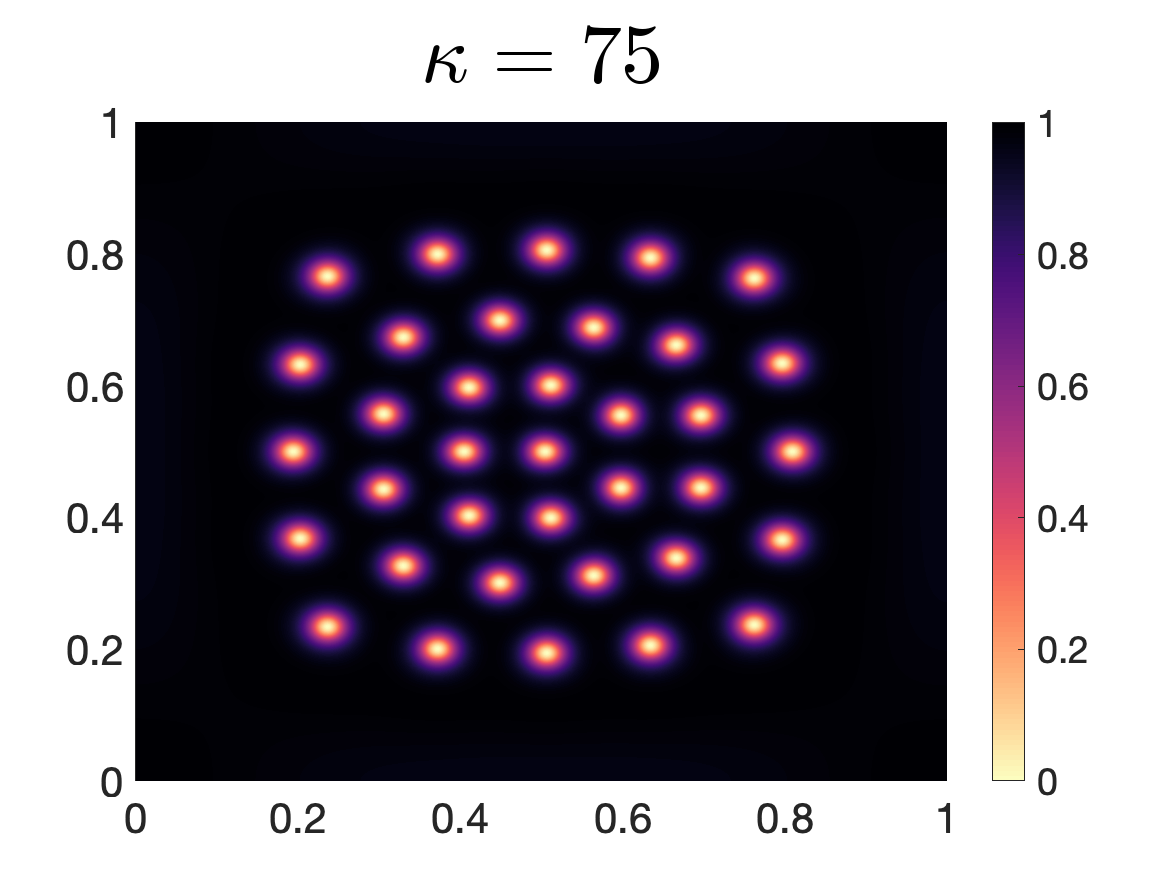}
\end{minipage}
\begin{minipage}{0.19\textwidth}
\includegraphics[scale=0.17]{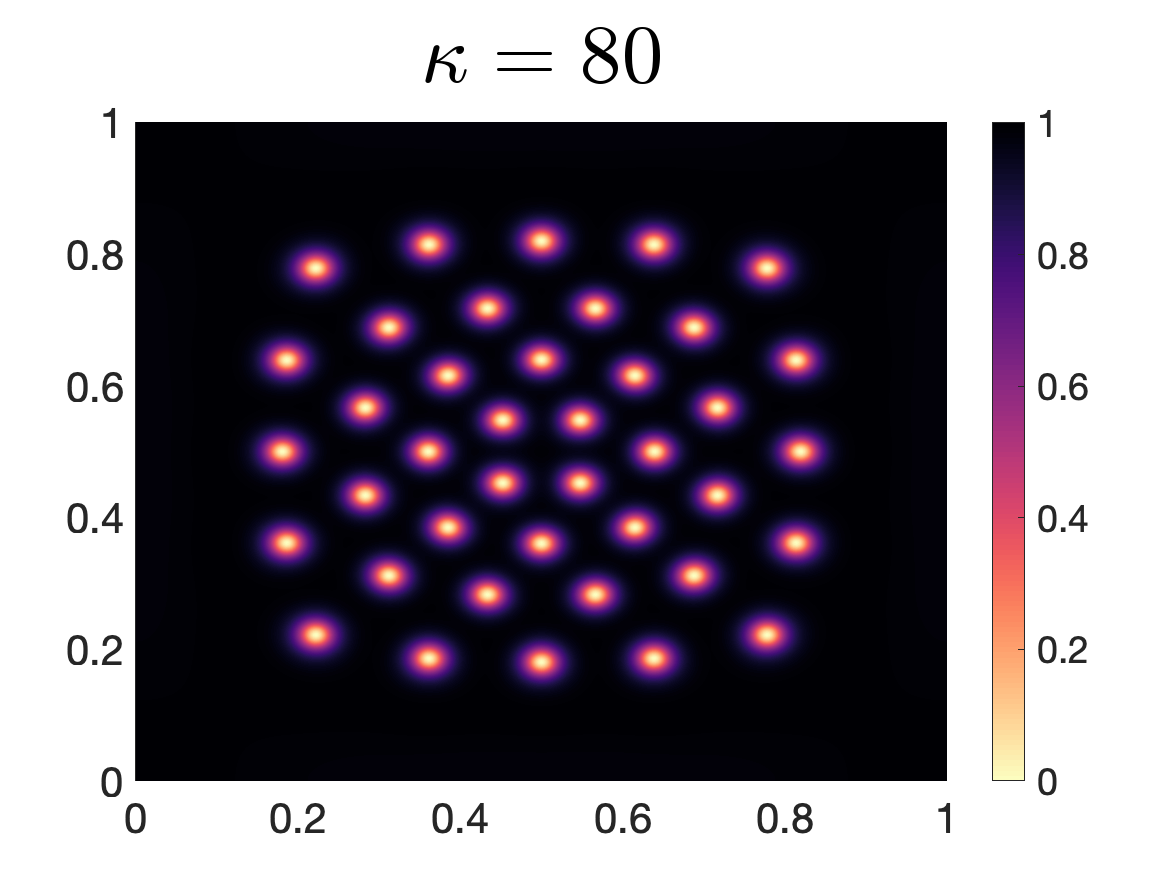}
\end{minipage} \\
\begin{minipage}{0.19\textwidth}
\includegraphics[scale=0.17]{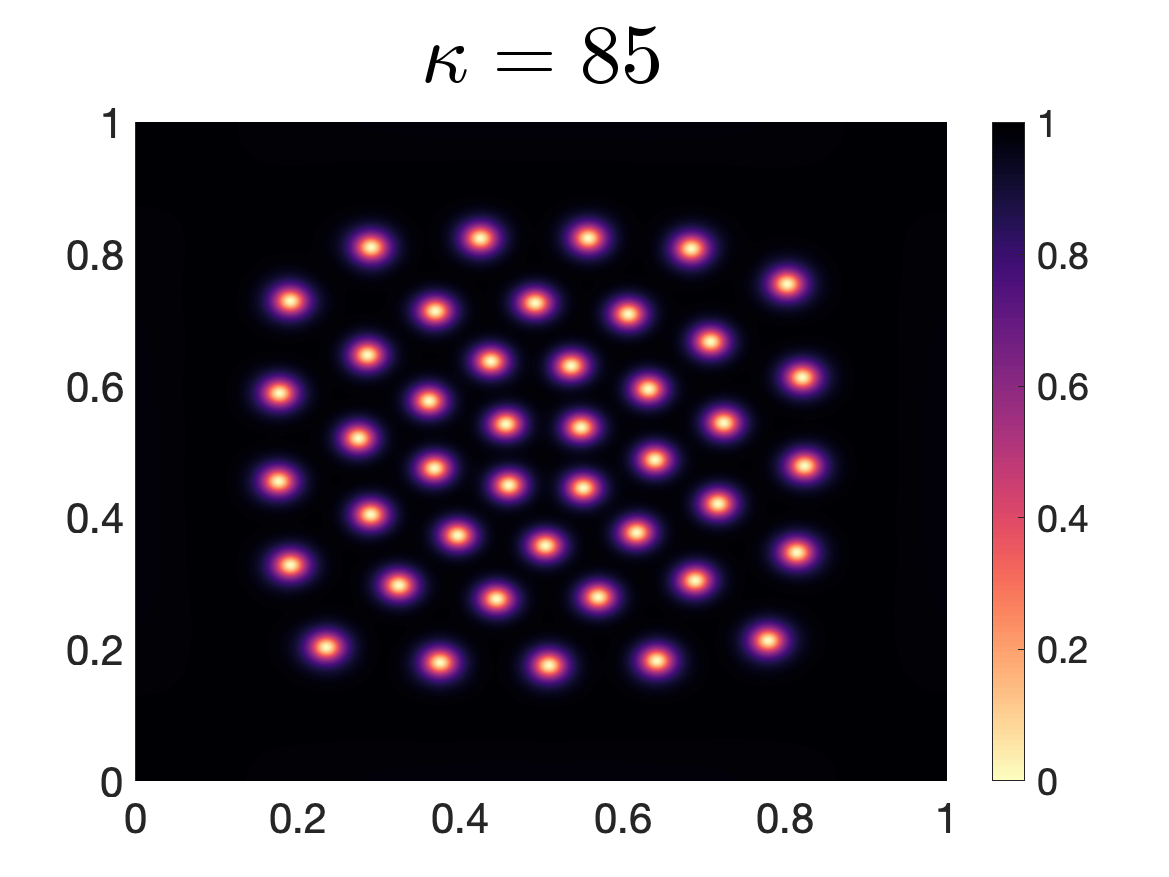}
\end{minipage}
\begin{minipage}{0.19\textwidth}
\includegraphics[scale=0.17]{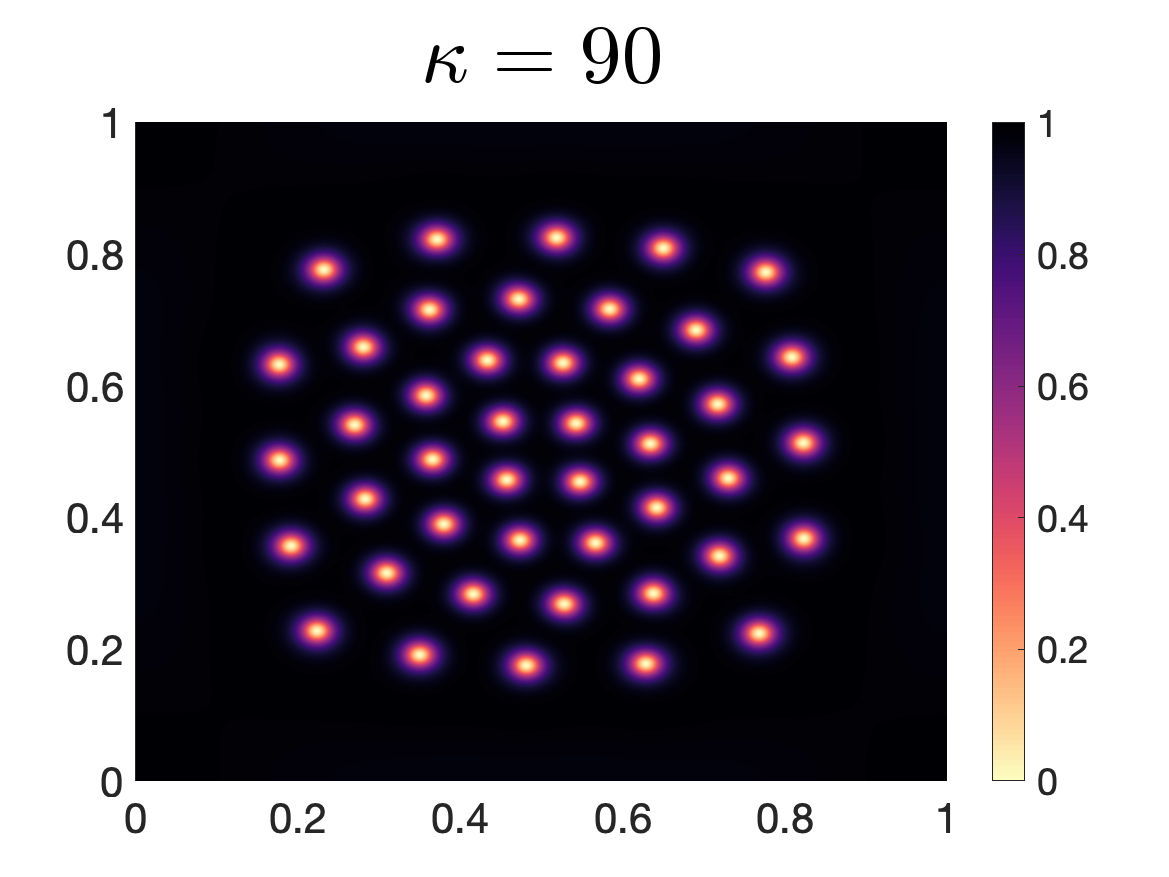}
\end{minipage}
\begin{minipage}{0.19\textwidth}
\includegraphics[scale=0.17]{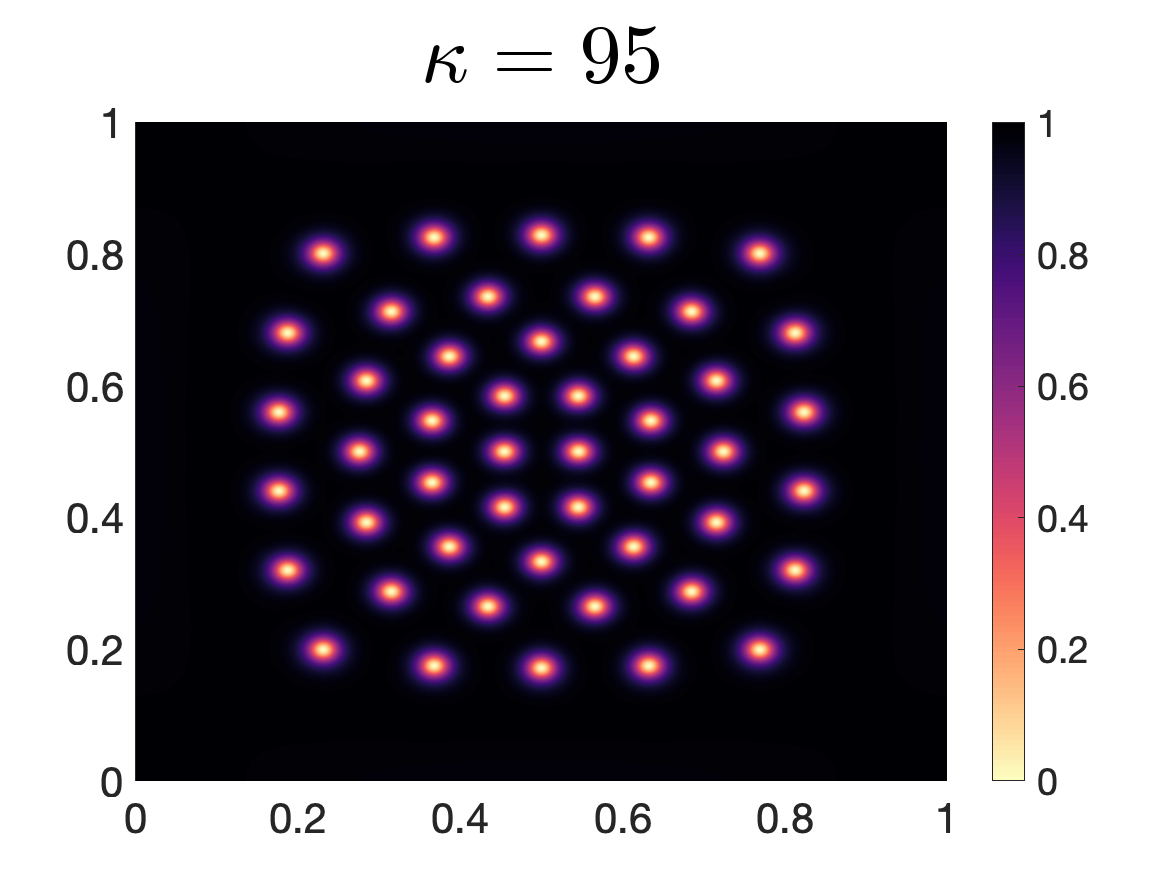}
\end{minipage}
\begin{minipage}{0.19\textwidth}
\includegraphics[scale=0.17]{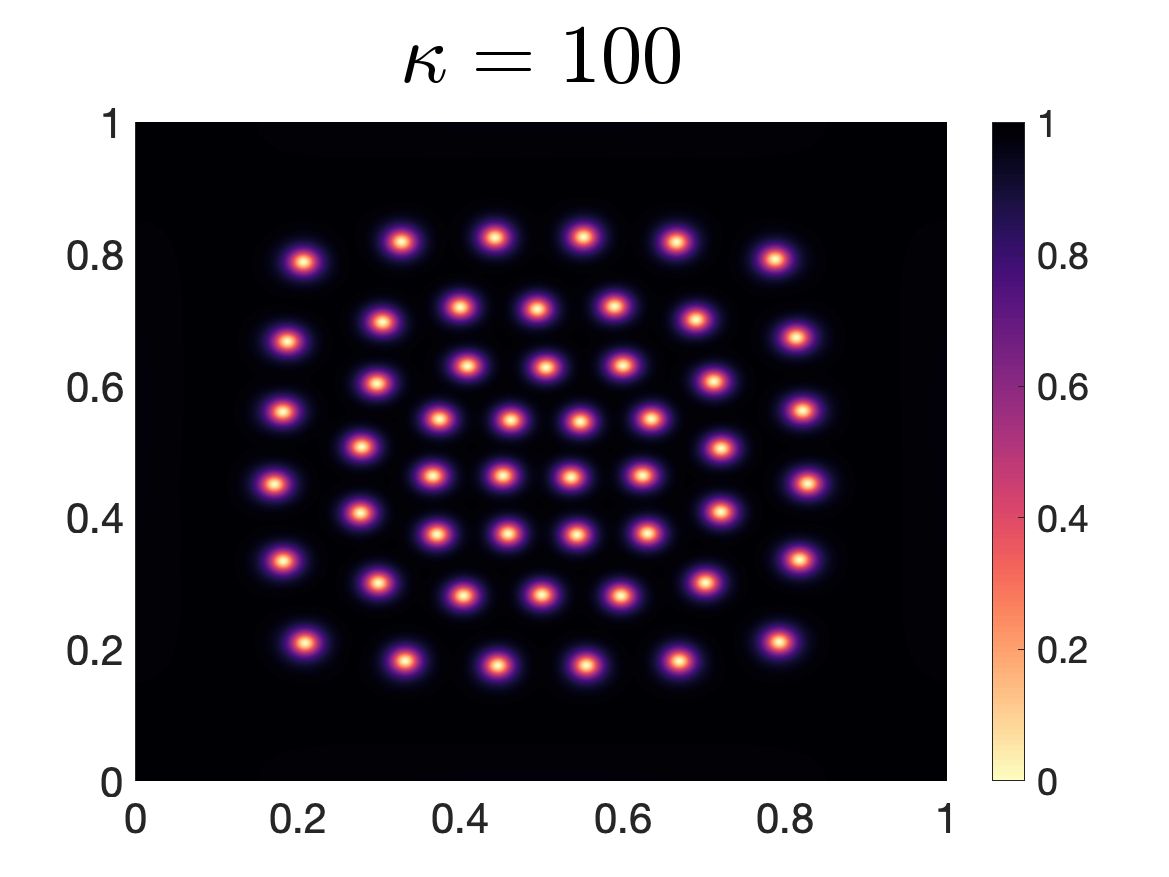}
\end{minipage}
\begin{minipage}{0.19\textwidth}
\phantom{\includegraphics[scale=0.17]{./min_kappa100.eps}}
\end{minipage}
	\caption{Densities $|u|^2$ of the computed discrete GL energy minimizers with $\A = \A_1$ and $\kappa = 5j$, $j =2,\dots,20$.}
	\label{fig:1}
\end{figure} 

\begin{figure}[h!]
\centering
\begin{minipage}{0.19\textwidth}
\includegraphics[scale=0.17]{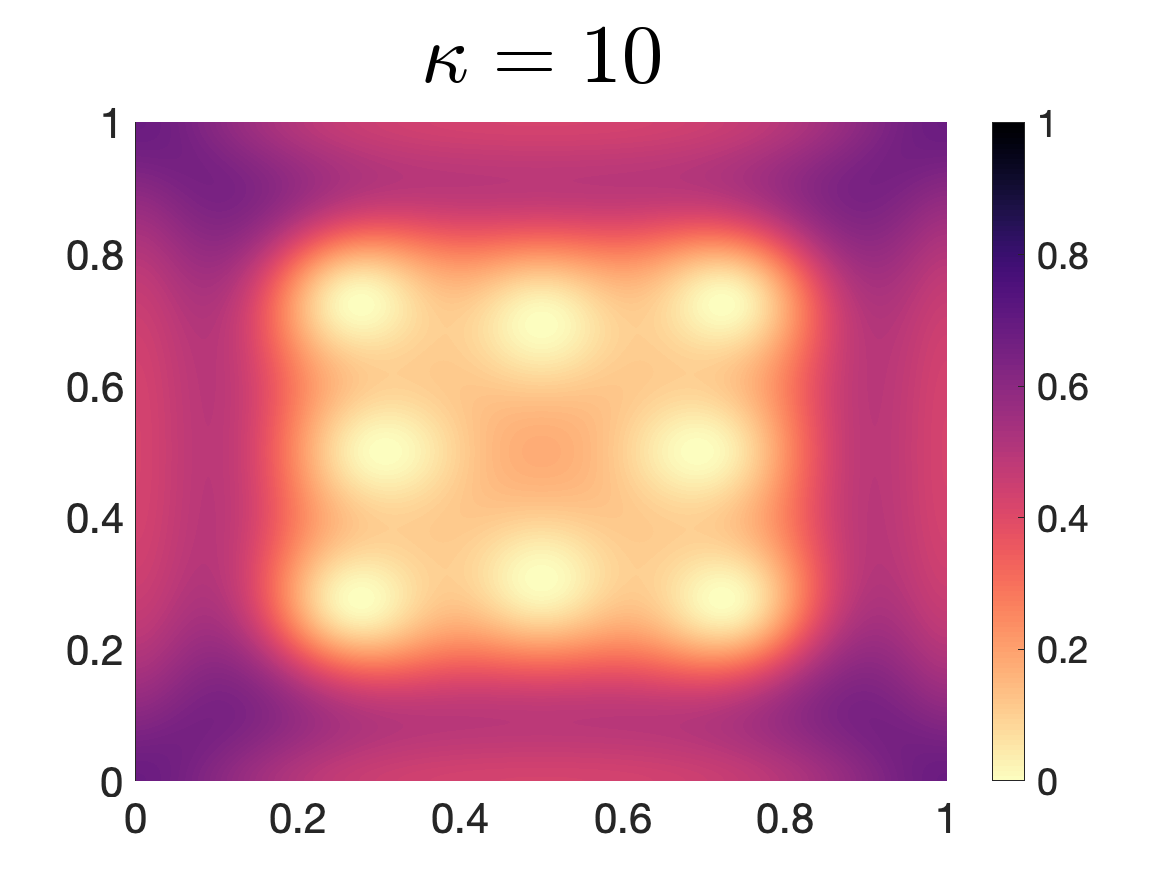}
\end{minipage}
\begin{minipage}{0.19\textwidth}
\includegraphics[scale=0.17]{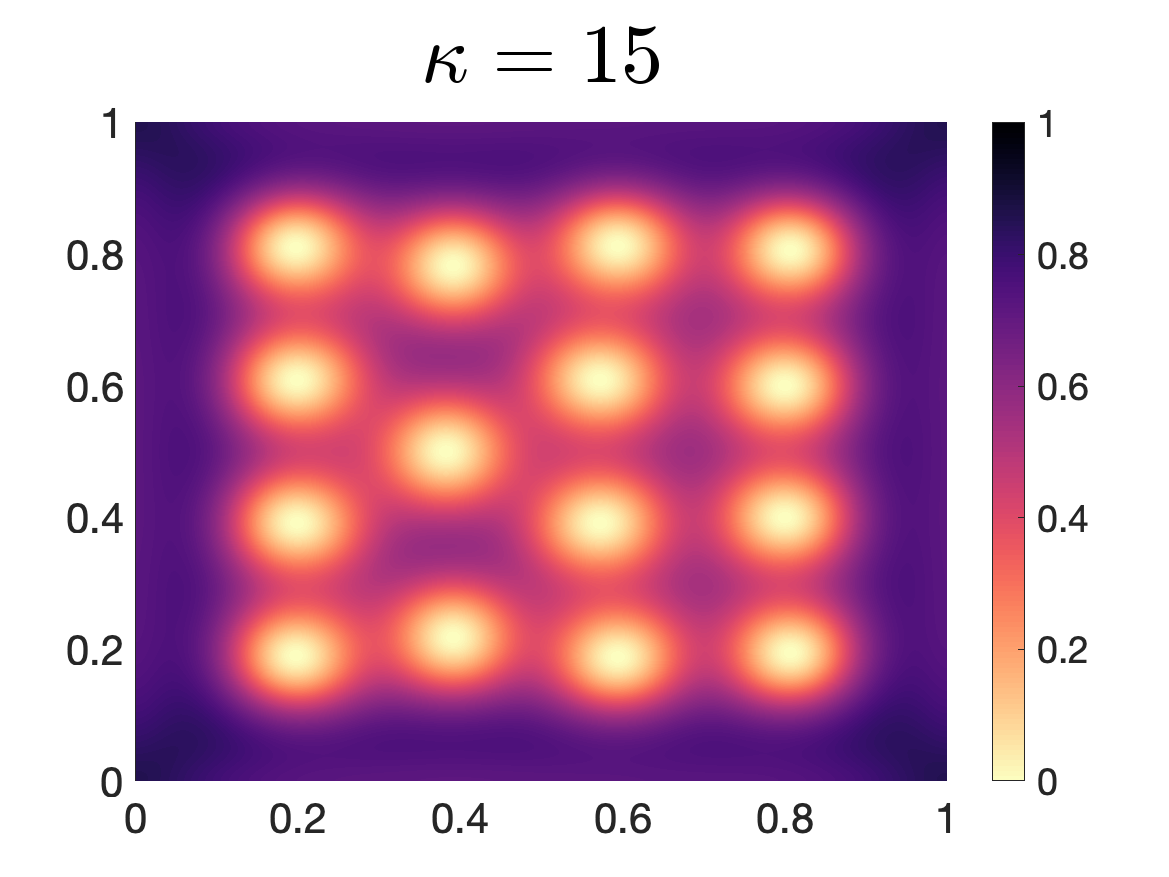}
\end{minipage}
\begin{minipage}{0.19\textwidth}
\includegraphics[scale=0.17]{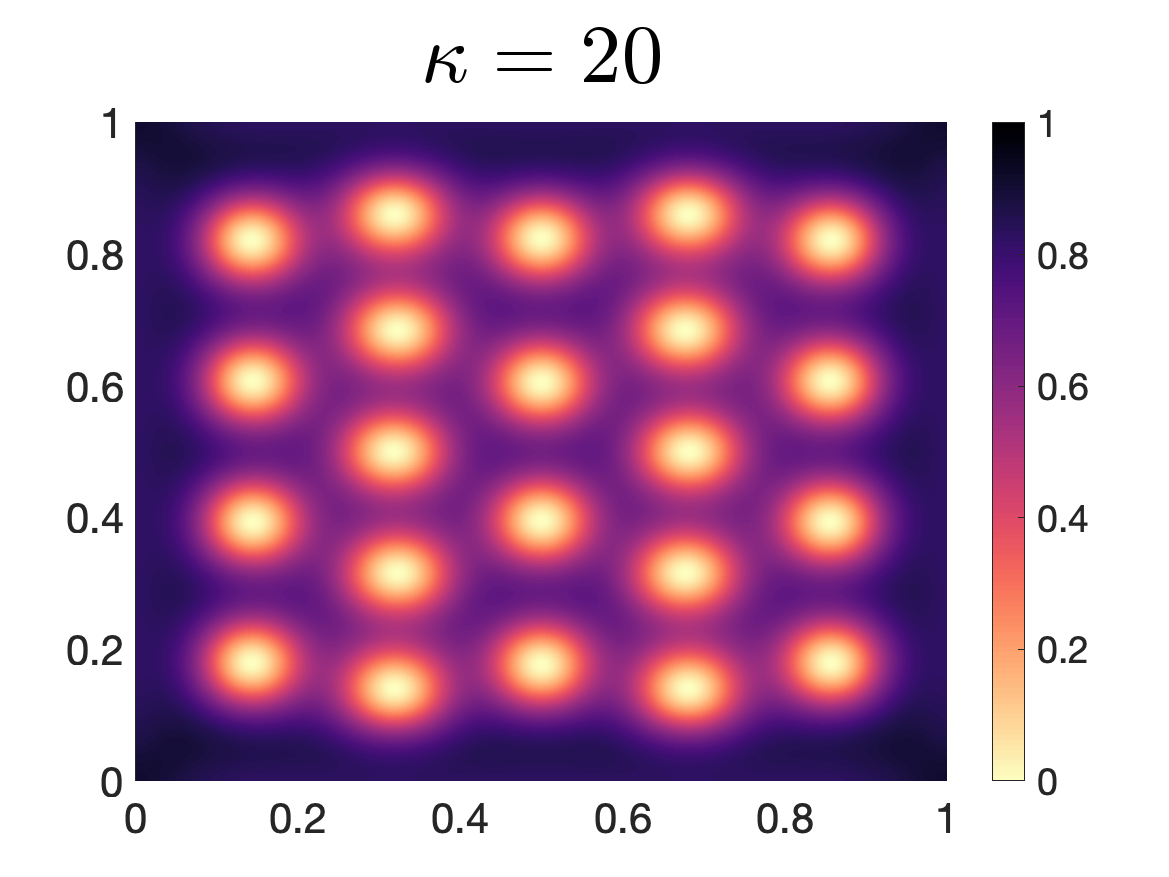}
\end{minipage}
\begin{minipage}{0.19\textwidth}
\includegraphics[scale=0.17]{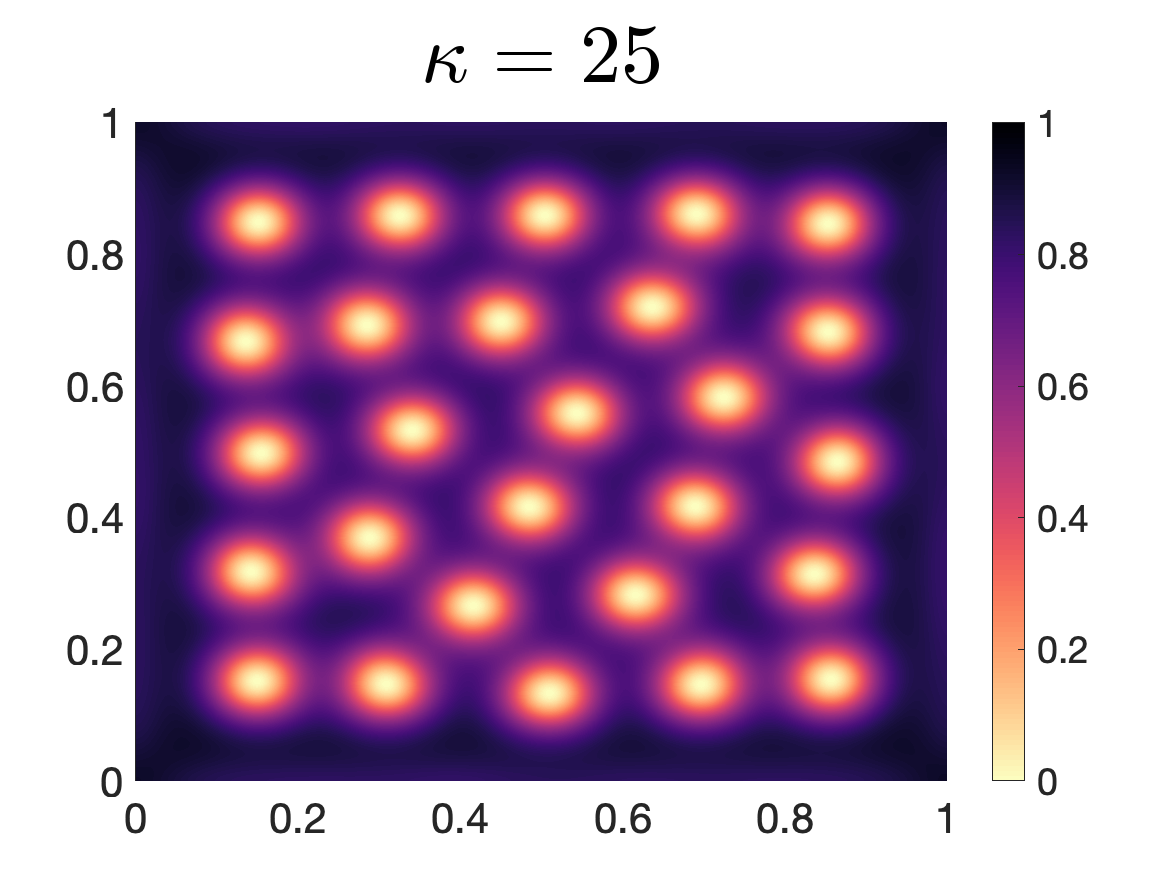}
\end{minipage}
\begin{minipage}{0.19\textwidth}
\includegraphics[scale=0.17]{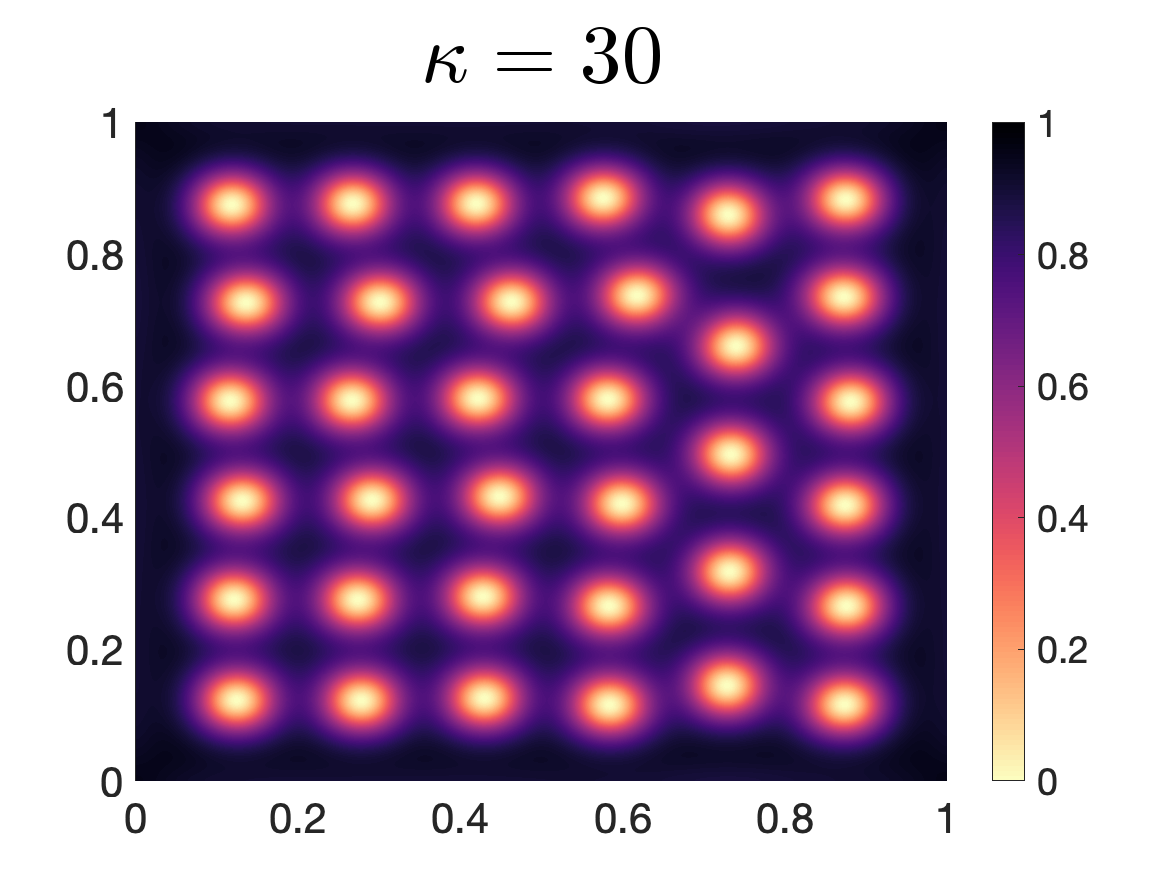}
\end{minipage} \\
\begin{minipage}{0.19\textwidth}
\includegraphics[scale=0.17]{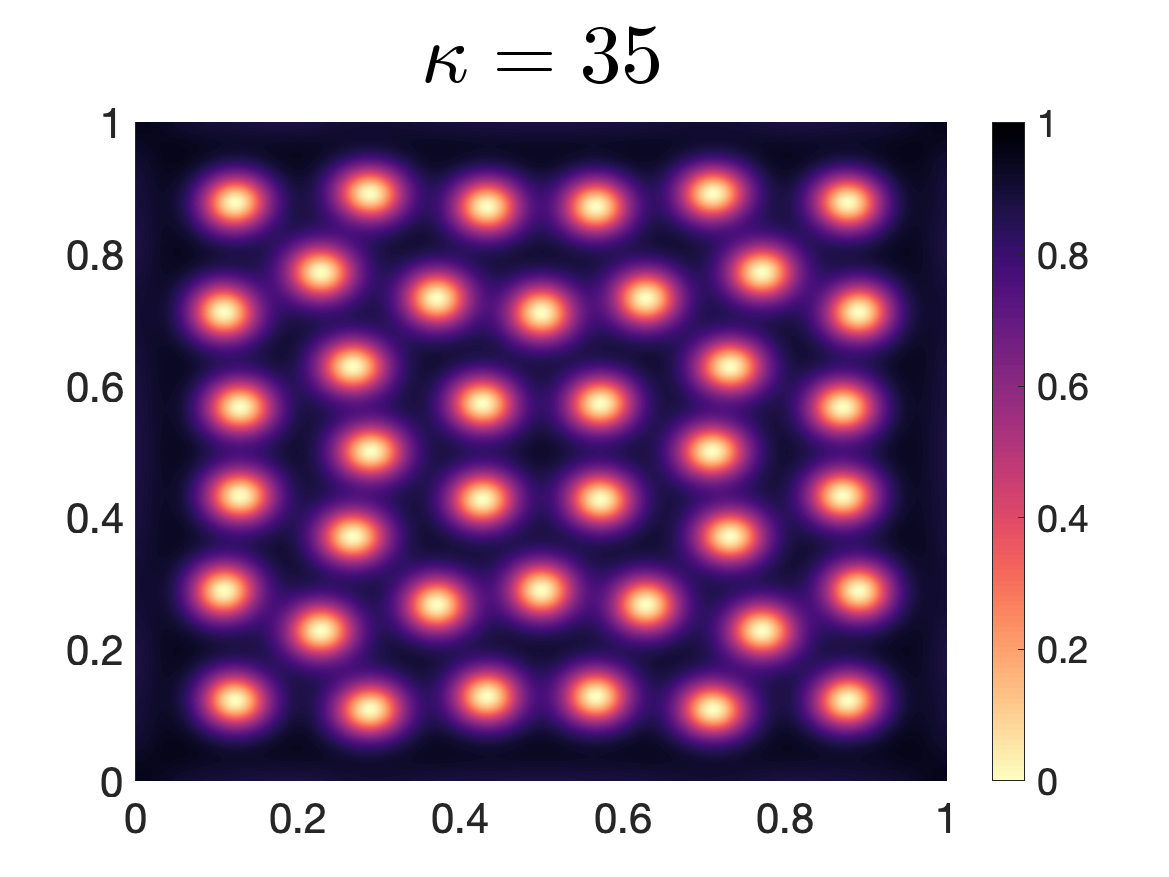}
\end{minipage}
\begin{minipage}{0.19\textwidth}
\includegraphics[scale=0.17]{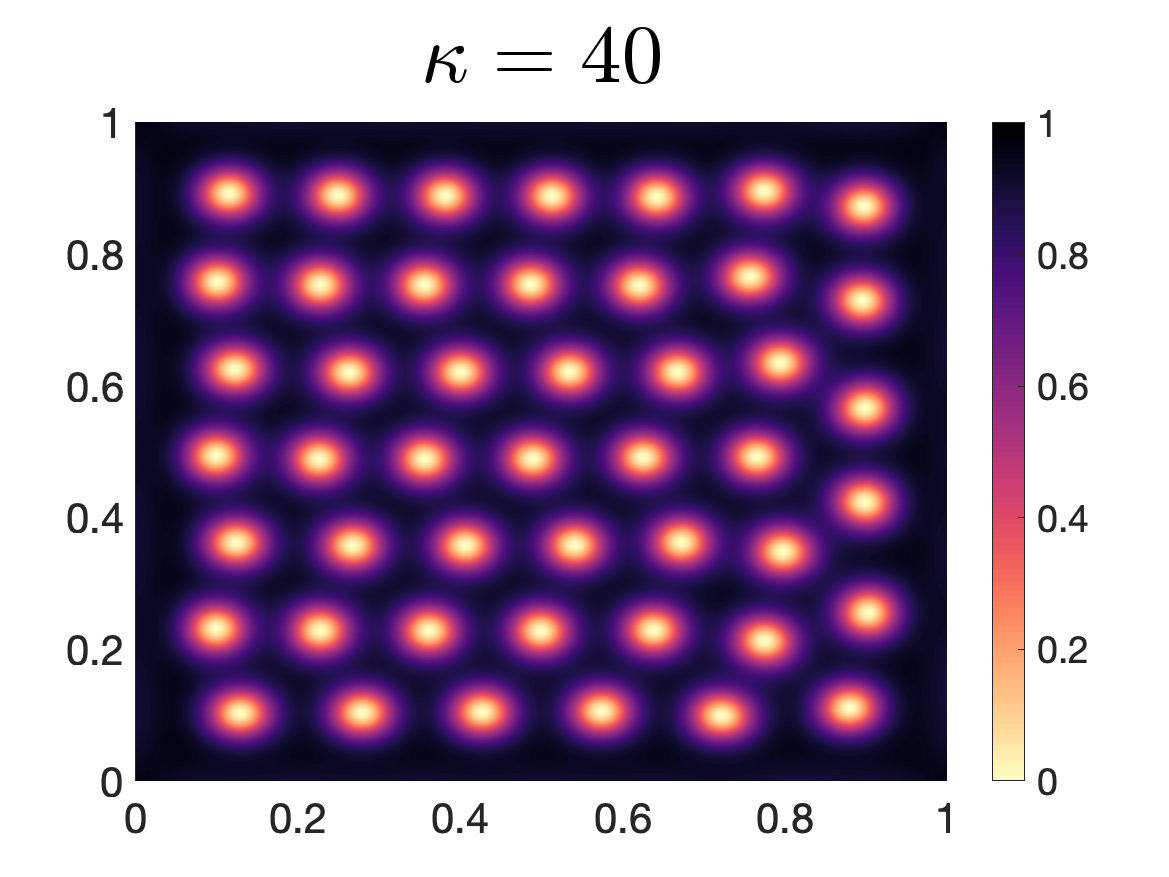}
\end{minipage}
\begin{minipage}{0.19\textwidth}
\includegraphics[scale=0.17]{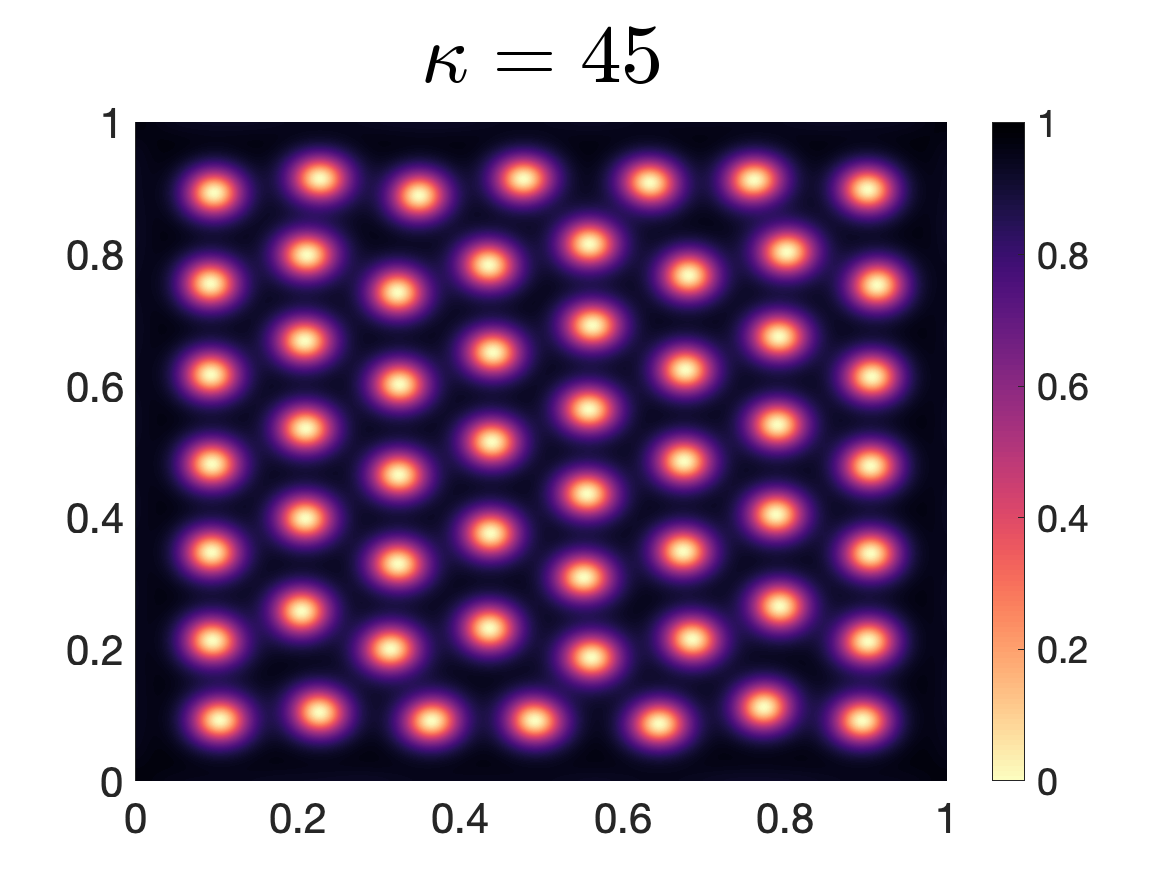}
\end{minipage}
\begin{minipage}{0.19\textwidth}
\includegraphics[scale=0.17]{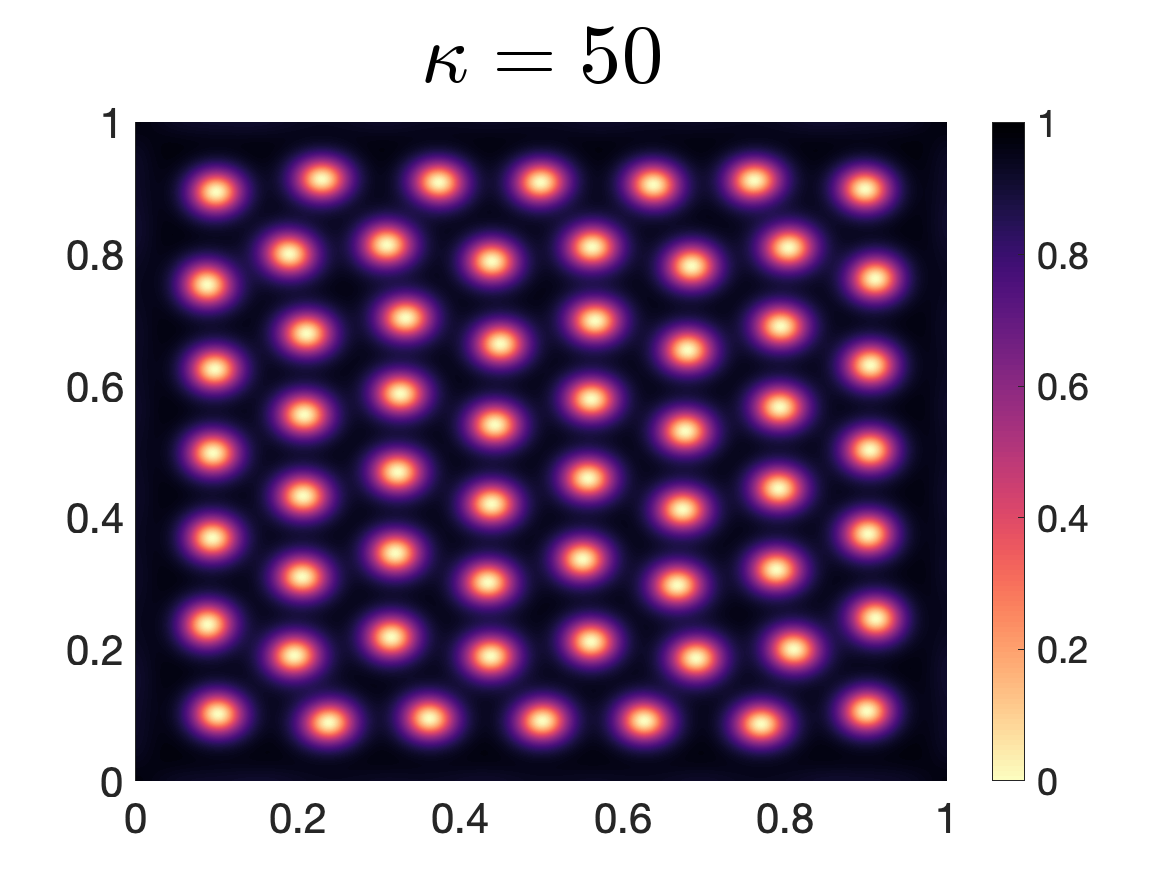}
\end{minipage}
\begin{minipage}{0.19\textwidth}
\includegraphics[scale=0.17]{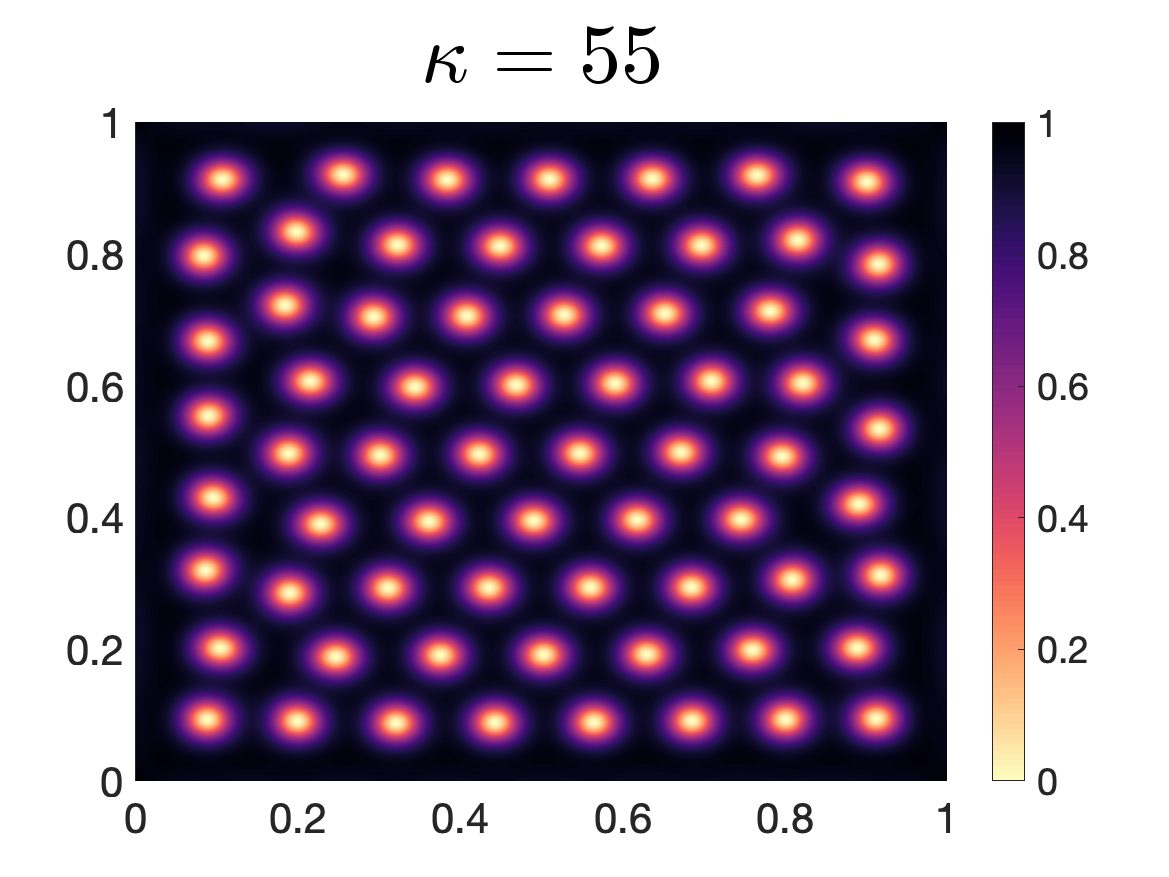}
\end{minipage} \\
\begin{minipage}{0.19\textwidth}
\includegraphics[scale=0.17]{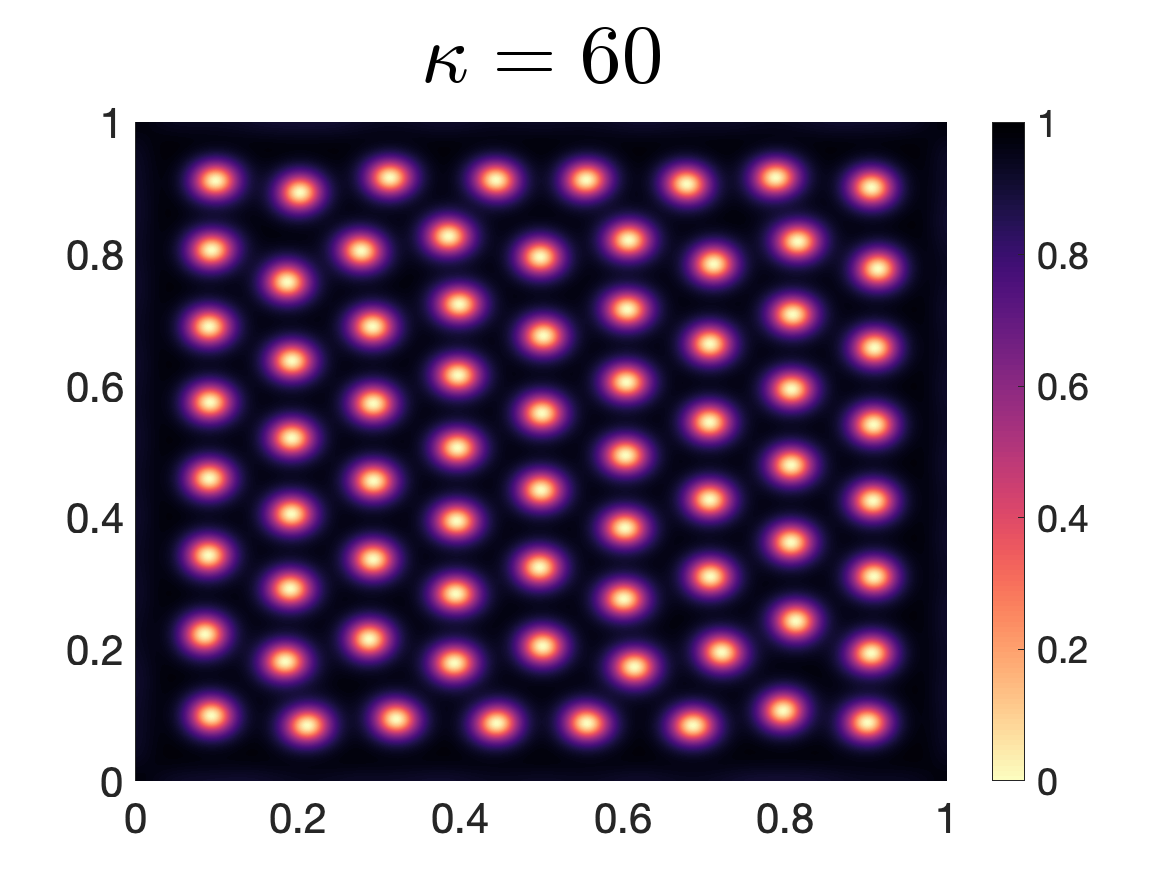}
\end{minipage}
\begin{minipage}{0.19\textwidth}
\includegraphics[scale=0.17]{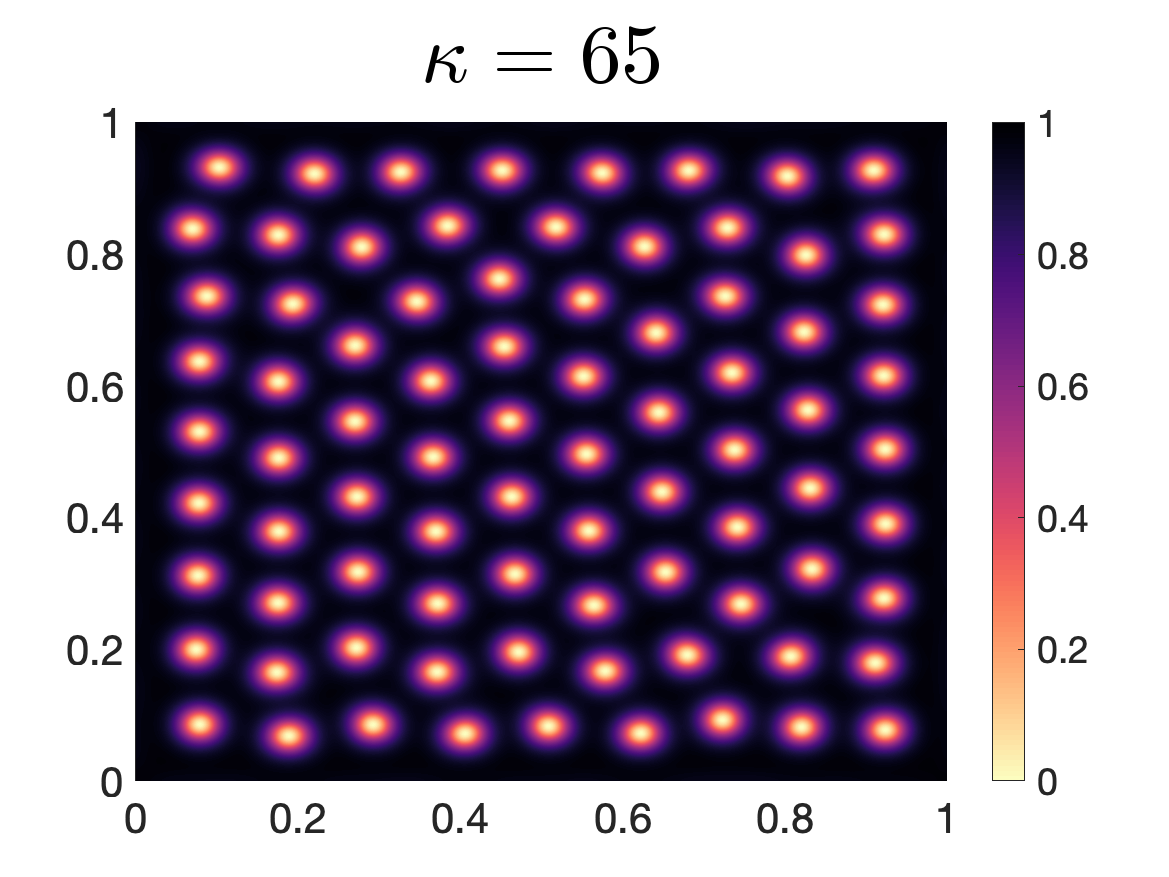}
\end{minipage}
\begin{minipage}{0.19\textwidth}
\includegraphics[scale=0.17]{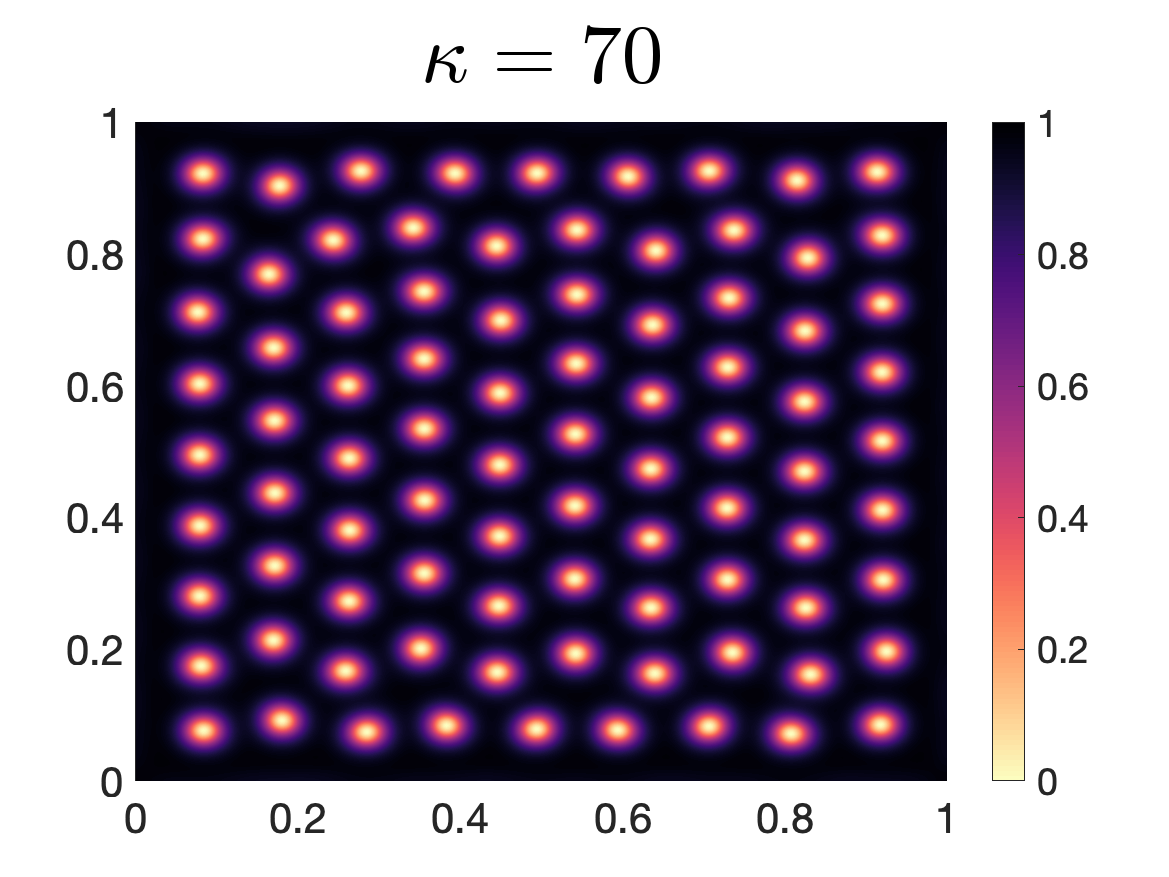}
\end{minipage}
\begin{minipage}{0.19\textwidth}
\includegraphics[scale=0.17]{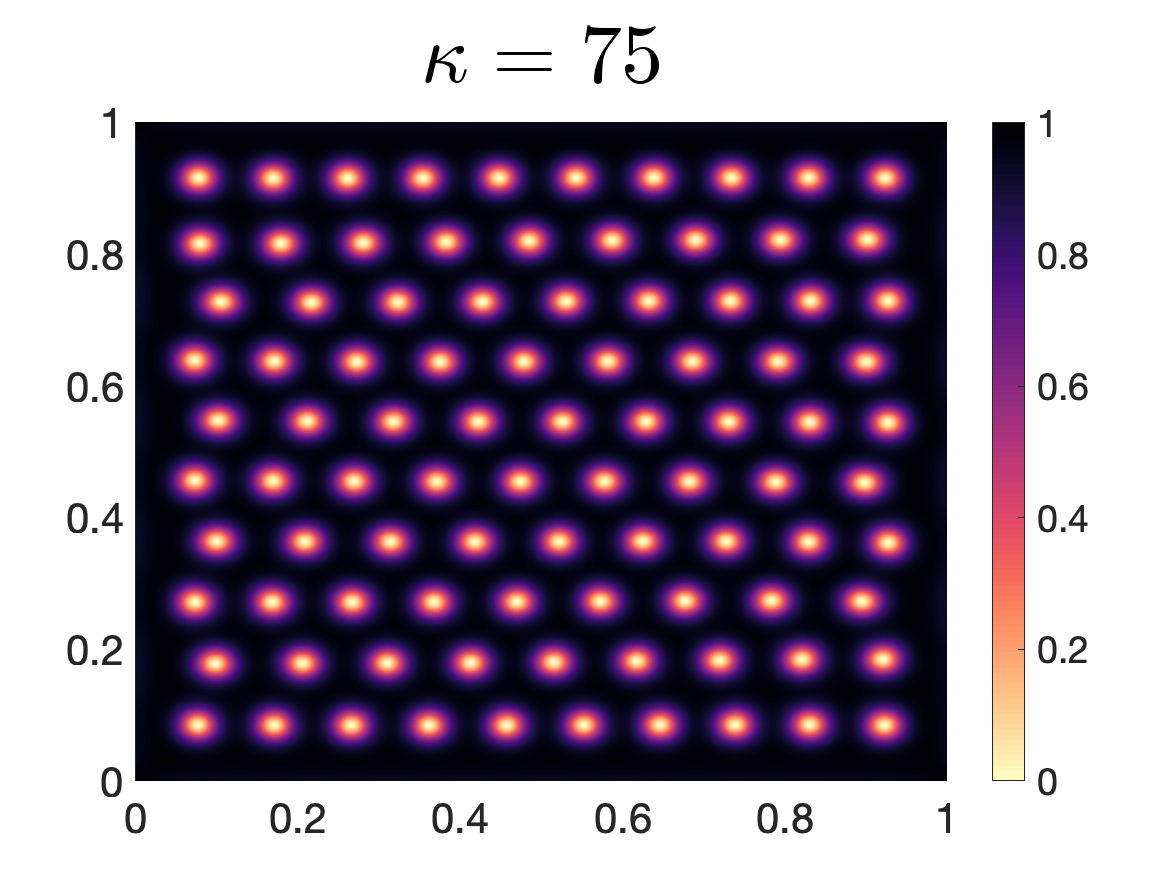}
\end{minipage}
\begin{minipage}{0.19\textwidth}
\includegraphics[scale=0.17]{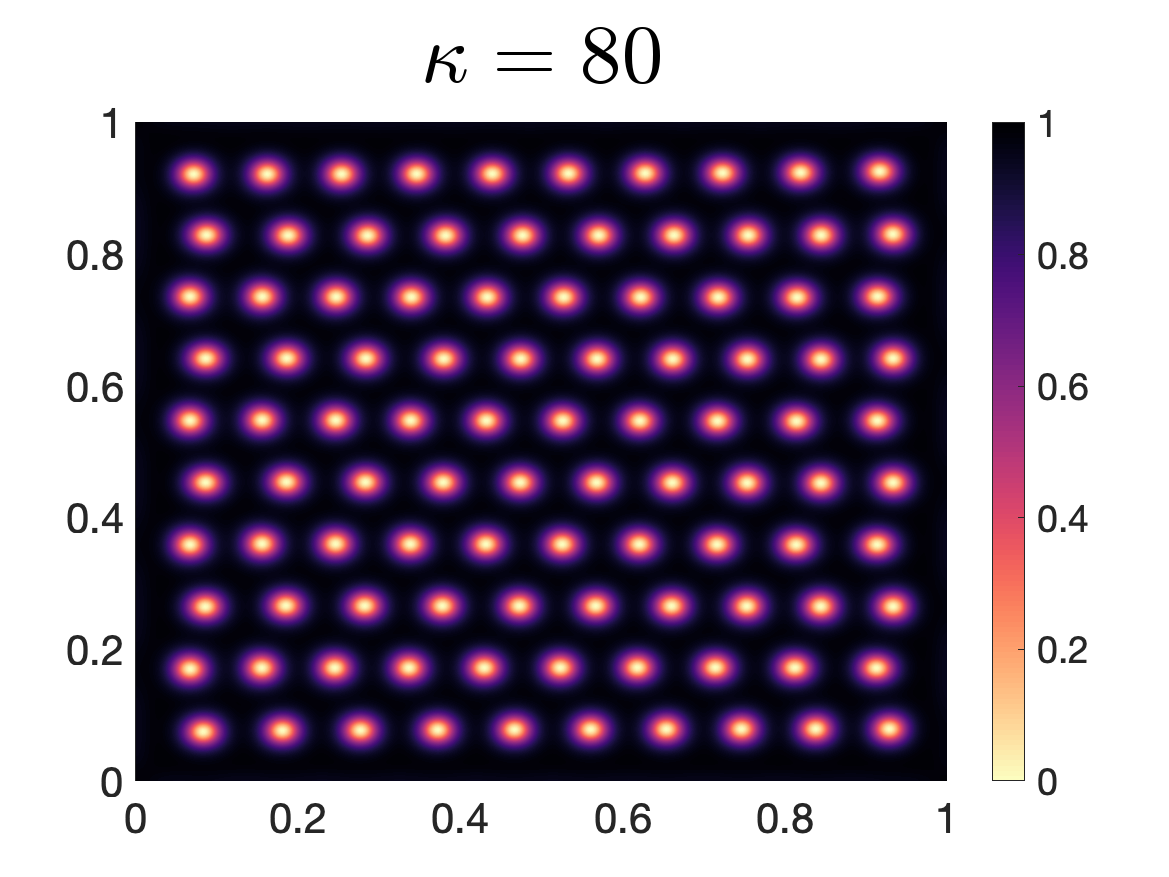}
\end{minipage} \\
\begin{minipage}{0.19\textwidth}
\includegraphics[scale=0.17]{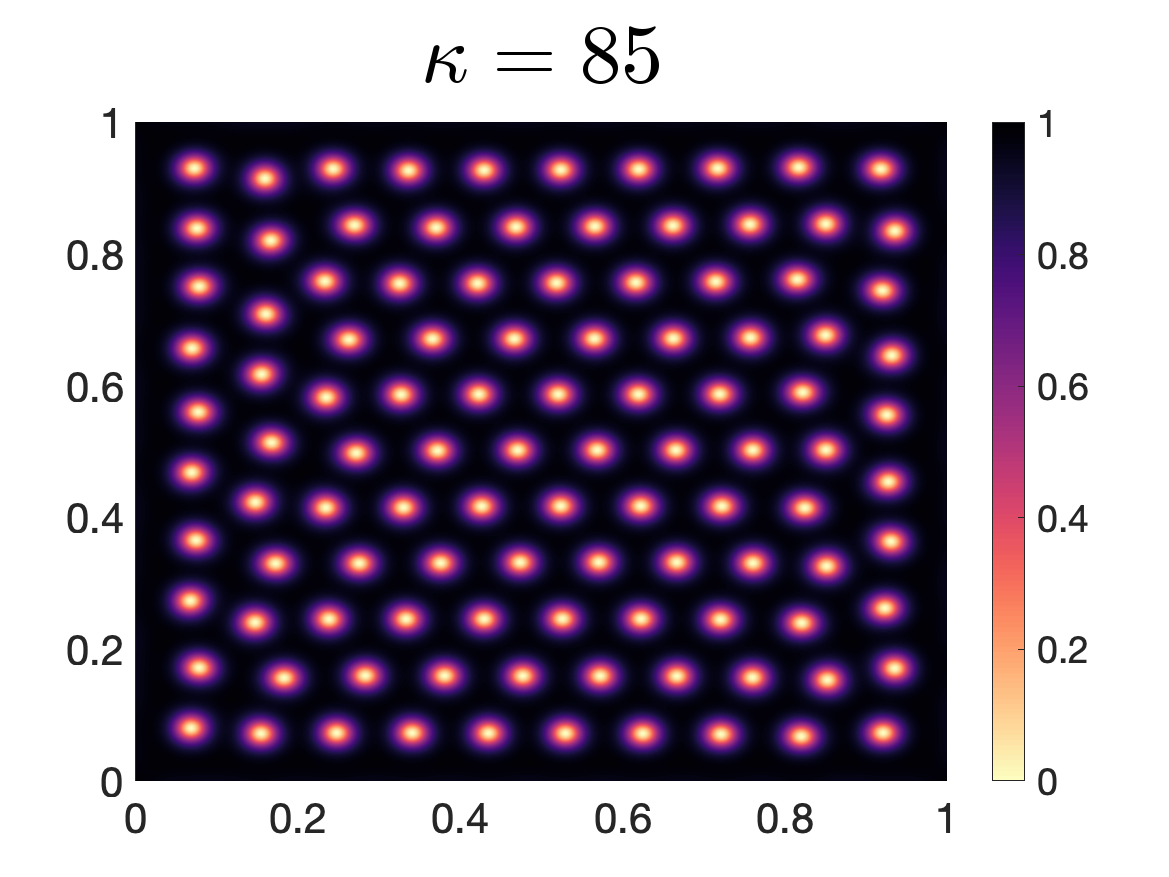}
\end{minipage}
\begin{minipage}{0.19\textwidth}
\includegraphics[scale=0.17]{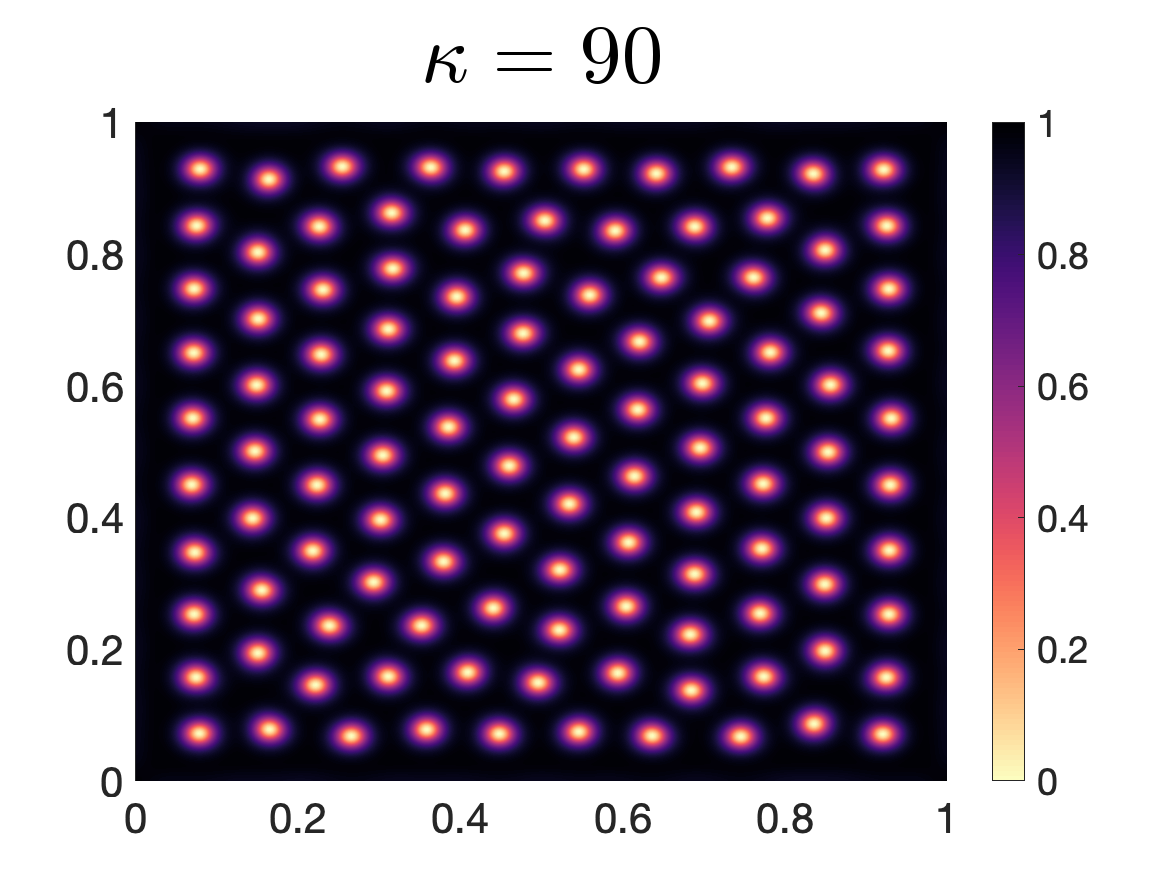}
\end{minipage}
\begin{minipage}{0.19\textwidth}
\includegraphics[scale=0.17]{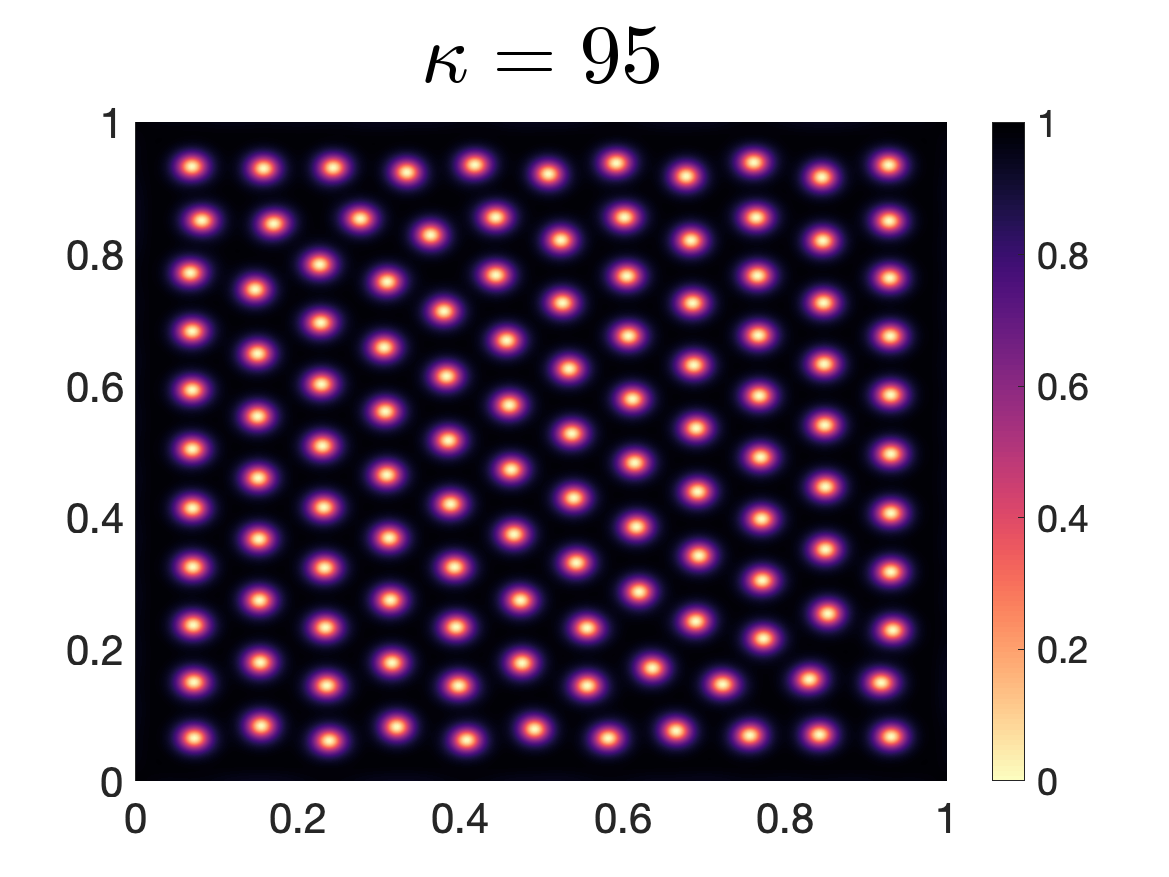}
\end{minipage}
\begin{minipage}{0.19\textwidth}
\includegraphics[scale=0.17]{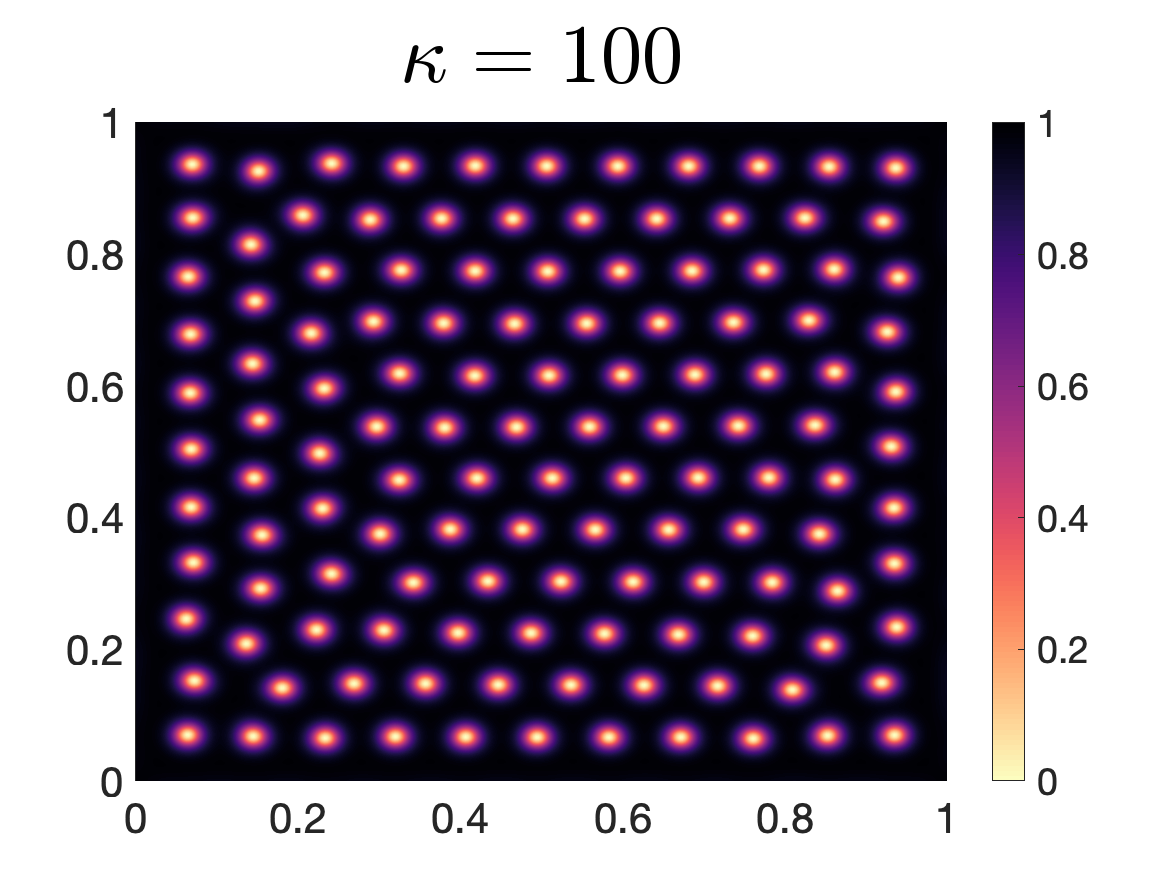}
\end{minipage}
\begin{minipage}{0.19\textwidth}
\phantom{\includegraphics[scale=0.17]{./min2_kappa100.eps}}
\end{minipage}
	\caption{Densities $|u|^2$ of the computed discrete GL energy minimizers with $\A = \A_2$ and $\kappa = 5j$, $j =2,\dots,20$.}
	\label{fig:2}
\end{figure} 

\begin{table}[h!]
\centering
\begin{tabular}{c cc cc}
\toprule
 & \multicolumn{2}{c}{$\boldsymbol{A}_1$} & \multicolumn{2}{c}{$\boldsymbol{A}_2$} \\
\cmidrule(lr){2-3} \cmidrule(lr){4-5}
$\kappa$ & $E(u)$ & $\lambda_2(\kappa)$ & $E(u)$ & $\lambda_2(\kappa)$ \\
\midrule
10  & 0.104590638 & 3.364695e-02 & 0.210302180 & 1.253325e-02 \\
15  & 0.087800780 & 5.043459e-04 & 0.175409470 & 1.277341e-03 \\
20  & 0.075291293 & 3.865896e-05 & 0.151924400 & 7.506412e-04 \\
25  & 0.064500846 & 4.387885e-05 & 0.133881341 & 3.254416e-04 \\
30  & 0.058934712 & 8.862484e-06 & 0.119779397 & 4.431788e-04 \\
35  & 0.052393817 & 5.051936e-05 & 0.109046329 & 1.906038e-04 \\
40  & 0.048082822 & 1.768819e-04 & 0.101158914 & 2.061991e-04 \\
45  & 0.045289744 & 1.287305e-04 & 0.093076130 & 7.584095e-05 \\
50  & 0.041412827 & 1.084707e-04 & 0.087120326 & 1.231738e-04 \\
55  & 0.038966075 & 4.098281e-05 & 0.081475140 & 1.364403e-04 \\
60  & 0.037183406 & 1.793039e-06 & 0.077086813 & 5.682032e-05 \\
65  & 0.034882249 & 5.122268e-07 & 0.073577591 & 6.054500e-05 \\
70  & 0.033180030 & 2.499888e-05 & 0.070531694 & 3.705614e-05 \\
75  & 0.031875279 & 3.501932e-05 & 0.066838003 & 4.146475e-05 \\
80  & 0.030282029 & 2.590564e-05 & 0.063537417 & 2.219414e-05 \\
85  & 0.029072447 & 4.891610e-07 & 0.061940527 & 2.156034e-05 \\
90  & 0.028146489 & 4.520968e-06 & 0.058974280 & 1.038614e-05 \\
95  & 0.026935874 & 1.156596e-05 & 0.057159507 & 1.165843e-05 \\
100 & 0.026059354 & 1.440620e-05 & 0.055481094 & 1.584083e-05 \\
\midrule
\multicolumn{1}{r}{} & \multicolumn{2}{c}{$\alpha \approx 2.726369$} & \multicolumn{2}{c}{$\alpha \approx 2.671856$} \\
\bottomrule
\end{tabular}
\caption{Energy levels $E(u)$ and spectral gap $\lambda_2(\kappa)$ of $E''(u)$ of the discrete minimizers for the two model problems with $\A = \A_1$ and $\A = \A_2$. The decay rate $\alpha$ of the spectral gap according to Conjecture \ref{conjecture} is calculated by least-squares regression.} \label{tab}
\end{table}

\begin{figure}[h!]
\centering
\begin{minipage}{0.45\textwidth}
\includegraphics[scale=0.4]{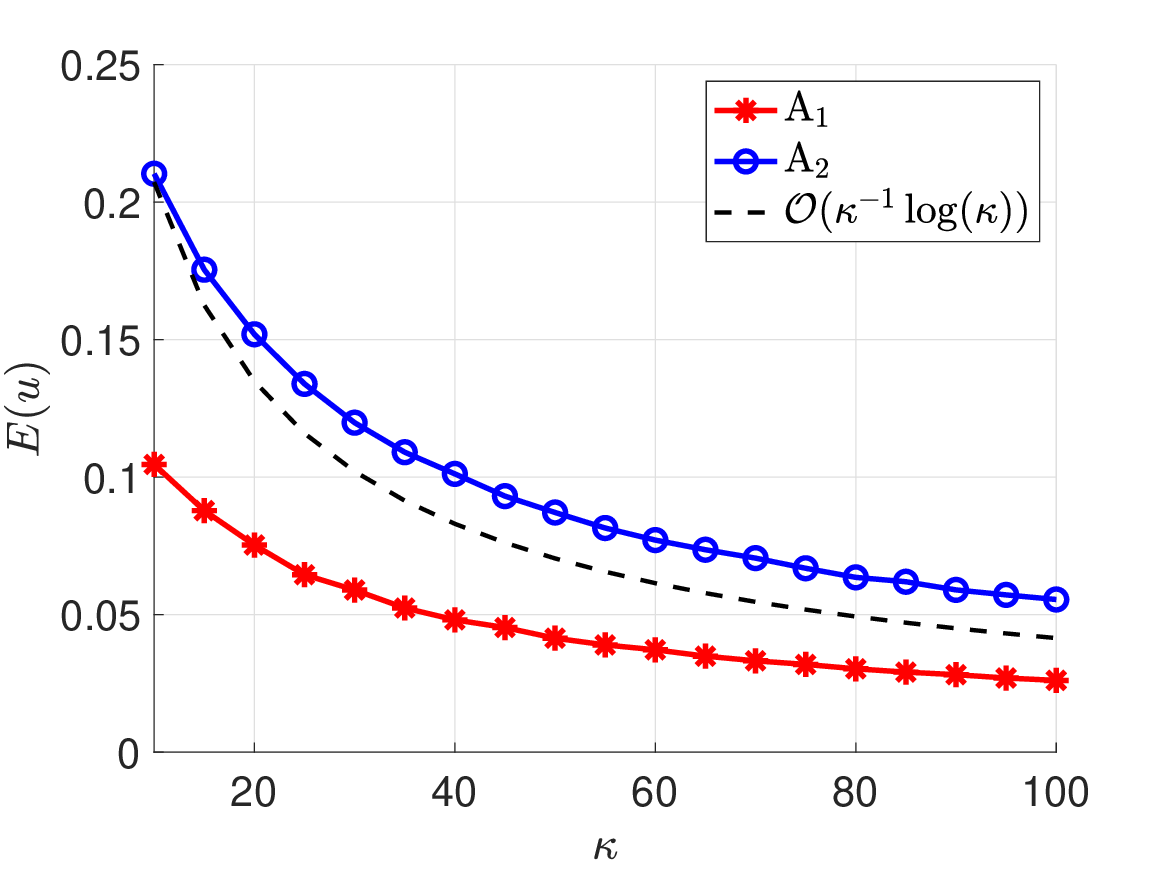}
\end{minipage}
\begin{minipage}{0.45\textwidth}
\includegraphics[scale=0.4]{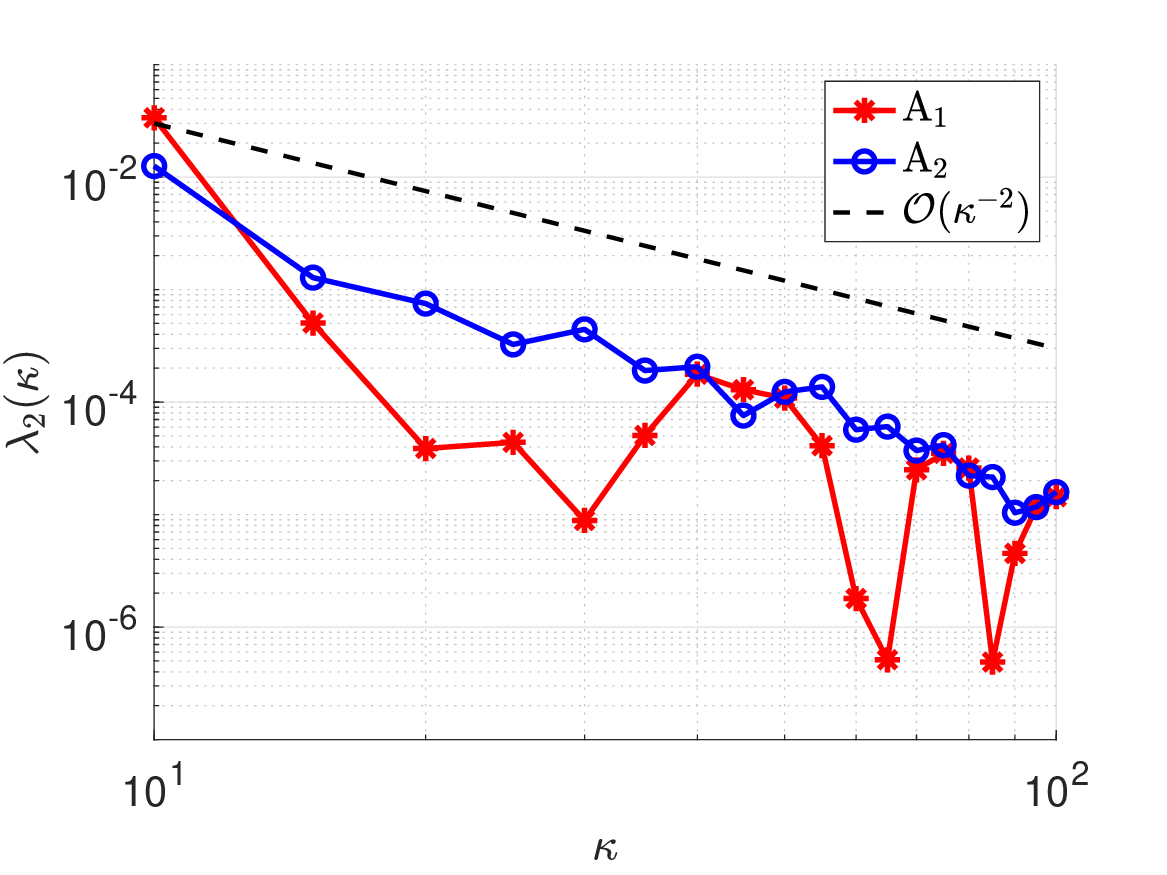}
\end{minipage}
	\caption{Energy levels $E(u)$ (left) and spectral gap $\lambda_2(\kappa)$ of $E''(u)$ for the computed discrete minimizers $u$ for both model problem $\A = \A_1$ and $\A = \A_2$.}
	\label{fig:3}
\end{figure} 

The densities $|u|^2$ of the discrete minimizers in both model settings for the varying values of $\kappa$ are depicted in Figure~\ref{fig:1} for $\mathbf{A} = \mathbf{A}_1$ and in Figure~\ref{fig:2} for $\mathbf{A} = \mathbf{A}_2$. The Abrikosov vortex lattice is clearly identifiable in the discrete minimizers. With increasing $\kappa$, the number of vortices increases while their size decreases. Comparing the two model problems, we observe that the vortex pattern is more centralized for $\mathbf{A} = \mathbf{A}_1$, which results from the fact that the vector potential decays to zero at the boundary. In contrast, for $\mathbf{A} = \mathbf{A}_2$, the magnetic field remains constant up to the boundary, and the vortex pattern is homogeneously distributed over the domain $\Omega$. Furthermore, in the latter case and for larger values of $\kappa$, the vortices align in a triangular/hexagonal lattice, consistent with the physical and mathematical theory for a constant magnetic field \cite{SaS07,Tinkham}.

In Table~\ref{tab} and Figure~\ref{fig:3} (left), we show the energy levels $E(u)$ of the computed discrete minimizers for varying values of $\kappa$. In particular, Figure~\ref{fig:3} (left) illustrates the decay of the energy with increasing $\kappa$. This decay can be identified with the rate $\mathcal{O}(\kappa^{-1} \log \kappa)$, which agrees with the asymptotic analysis in \cite[Theorem 8.1]{SaS07} in the high-$\kappa$ regime.

Finally, we use this data to verify Conjecture~\ref{conjecture}. For each discrete minimizer $u$, we compute the spectrum of the second Fr\'echet derivative, and in particular the spectral gap, given by the second smallest eigenvalue $\lambda_2(\kappa)$. The size of the spectral gap is shown in Figure~\ref{fig:3} (right), and the particular values are reported in Table~\ref{tab} for both model problems. Recall that the coercivity constant from Lemma \ref{lem:coe} we are interested in is bounded by the spectral gap, i.e., $\rho(\kappa)^{-1} \lesssim \lambda_2(\kappa)$.

As a function of $\kappa$, the spectral gap clearly decays, leading to a corresponding decay of the coercivity constant $\rho(\kappa)^{-1}$, as predicted by Conjecture~\ref{conjecture}. Up to small pollution, this decay can be described by $\lambda_2(\kappa) \ls \kappa^{-\alpha}$. The rate $\alpha$ can be estimated using a least-squares regression. For the first model problem ($\A = \A_1$), we calculate $\alpha \approx 2.726369$, whereas for the second model problem with a homogeneous magnetic field ($\A = \A_2$), we find $\alpha \approx 2.671856$. In both cases, a rough estimate of the decay rate yields $\alpha \approx 2$ for which we therefore have shown strong numerical evidence.

In summary, these results provide numerical evidence supporting Conjecture~\ref{conjecture}, indicating that the decay rate is at least of order $2$. Consequently, we obtain the overall estimate $\rho(\kappa) \ls \lambda_2(\kappa) \ls \kappa^{-2}$, which is confirmed numerically.

\textbf{Acknowledgement.} This work was funded by the Deutsche Forschungsgemeinschaft (DFG, German Research Foundation) under Germany's Excellence Strategy – EXC-2047/1 – 390685813.

\end{document}